\documentclass[12pt,a4paper]{amsart}
\usepackage{amssymb,graphicx}
\usepackage{graphicx}
\usepackage[dvips]{epsfig}

\numberwithin{equation}{section}
\numberwithin{figure}{section}

\def\Z{\mathbb{Z}}

\def\<{\langle}
\def\>{\rangle}

\def\S{\mathcal{S}}

\renewcommand{\hat}{\widehat}
\renewcommand{\tilde}{\widetilde}

\newtheorem{lem}{Lemma}
\newtheorem{thm}[lem]{Theorem}
\newtheorem{cor}[lem]{Corollary}

\newenvironment{pfof}[1]{\par\medskip\noindent{\em Proof of #1}.}{\hfill $\square$\par\medskip}

\title{On singular equations over torsion-free groups}
\author{Martin Edjvet}
\address{School of Mathematical Sciences\\ University of Nottingham\\ University Park\\ Nottingham NG7 2RD\\ UK}
\email{martin.edjvet@nottingham.ac.uk}
\author{James Howie}
\address{Maxwell Institute and School of Mathematical and Computer Sciences\\ Heriot Watt University\\ Edinburgh EH14 4AS\\ UK}
\email{J.Howie@hw.ac.uk}

\thanks{The second author was supported in part by 
Leverhulme Trust Emeritus Fellowship EM-2018-023$\backslash$9 }

\begin{document}

\begin{abstract}
We prove a Freiheitssatz for one-relator products of torsion-free groups, where the relator has syllable length at most $8$.  This result has applications to equations over torsion-free groups: in particular a singular equation of syllable length at most $18$ over a torsion-free group has a solution in some overgroup.
\end{abstract}

\maketitle

\section{Introduction}

An {\em equation} over a group $A$ in an {\em indeterminate} $t$ is just an expression
$w(t)=1$, where $w=w(t)$ is a word in the free product $A*\<t\>$.  This equation
has a solution in $A$ (resp. in an overgroup $G$ of $A$) if there is an element
$h\in A$ (resp. $h\in G$) such that substituting $h$ for $t$ in $w$ and evaluating in $A$ (resp. $G$) gives the identity.  It is well-known that a solution for $w(t)=1$ exists
in some overgroup if and only if the natural map from $A$ to $G:=(A*\<t\>)/N(w)$ is
injective (where $N(w)$ denotes the normal closure of $w$ in $A*\<t\>$) -- in which case
$G$ may be taken to be the overgroup in question, and the coset $t.N(w)$ to be the element
$h$ in the definition.

Thus there is a natural connection between the study of equations over groups and
that of {\em one-relator products}.  A one-relator product of groups $A_\lambda$
($\lambda\in\Lambda$) is the quotient $G$ of the free product $\ast_{\lambda\in\Lambda} A_\lambda$ by the normal closure of a single element $w$.  A one-relator product
$G$ of groups $A_\lambda$ is said to satisfy the {\em Freiheitssatz} if each $A_\lambda$
embeds in $G$ via the natural map.  This idea generalises the classical Freiheitssatz
of Magnus \cite{Mag}.  In the case where $\{A_\lambda,\lambda\in\Lambda\}=\{A,\<t\>\}$,
the Freiheitssatz says that the equation $w(t)=1$ has a solution in $G$, and moreover
the solution element $h=t.N(w)$ has infinite order in $G$.

Equations over groups and one-relator products have been studied extensively by various authors over a long period.  (See, for example, \cite{CH}, \cite{DH} and \cite{FR}.)

In the present article we consider equations over torsion-free groups.  The most striking
result here is that of Klyachko \cite{Kl}, that any equation with {\em exponent sum} $\pm 1$
in the indeterminate has a solution in an overgroup.  If we write
$$w=a_1t^{m(1)}\cdots a_kt^{m(k)}\in A*\<t\>$$
as a cyclically reduced word with $a_i\in A$ and $m(i)\in\Z$ for each $i$, then by definition the {\em exponent sum}
is $\sum_{i=1}^k m(i)$.  An equation with exponent sum $0$ is called {\em singular};
one with non-zero exponent sum is called {\em non-singular}.  In general it seems to be
easier to prove results for non-singular equations than for singular ones.
For example if $k=4$ and
$m(i)=1$ ($1 \leq i \leq 4$) then there is always a solution \cite{Lev}; if $m(j)=1$ ($1 \leq j \leq 3$) and $m(4)=-1$ then there is always a solution \cite{EH};
if $m(1)=m(2)=1$ and $m(3)=m(4)=-1$ then there is a solution provided one assumes that 
$a_1^2 \neq 1$, $a_3^2 \neq 1$ and $a_1 a_2 \neq 1$ \cite{E}; or if $m(1)=m(3)=1$ and $m(2)=m(4)=-1$ the problem remains very much open.

Our interest in the present article is more focussed on singular equations, in the spirit
of \cite{BH} (see also \cite{BE}, \cite{IK}, \cite{Kim} and \cite{Pri}).  Our principal result is the following.

\begin{thm}\label{main}
Let $A,B$ be torsion-free groups, and let $$w=a_1b_1...a_kb_k \in A*B$$
where $a_i\in A$ and $b_i\in B$ for each $i=1,...,k$ and where the $a_i$ and the $b_i$ are non-trivial.  If $k\le 4$ then each of $A,B$ embeds in
$$G:=(A*B)/N(w)$$
via the natural map.
\end{thm}

This result has some obvious consequences for the study of
equations over torsion-free groups.  The first of these is an
extension of \cite[Theorem 2(iv)]{BH} from $k\le 3$ to $k\le 4$,
and is obtained by putting $B=\<t\>\cong\Z$ in Theorem \ref{main}.

\begin{cor}\label{c1}
Let $G=(A*\<t\>)/N(w)$, where $A$ is torsion-free and let
$$w=a_1t^{m(1)}\cdots a_kt^{m(k)}\in A*\<t\>$$
where the $a_i$ are non-trivial elements of $A$,each $m(i) \neq 0$ and
$k\le 4$.  Then the natural maps
$A\to G$ and $\<t\>\to G$ are injective.
\end{cor}

The next result specifically addresses the solubility of certain
{\em singular} equations, that is equations $w(t)=1$ in which
the exponent sum of the variable $t$ in the word $w$ is equal to $0$.

\begin{cor}\label{c2}
Let
$$w=a_1t^{m(1)}\cdots a_kt^{m(k)}\in A*\<t\>$$
be a word where the $a_i$ are non-trivial elements of $A$ and each $m(i) \neq 0$ 
such that $\sum_{i=1}^k m(i)=0$ and such that the sequence
of partial sums $\left(\sum_{i=1}^j m(i)\right)_{j=1}^k$ attains
its maximum value at most four times and its minimum value at most four times.  Then the natural map
$$A\to G:=(A*\<t\>)/N(w)$$
is injective.
\end{cor}

Recall that the {\em syllable length} of a cyclically reduced word
$w=a_1t^{m(1)}\cdots a_kt^{m(k)}\in A*\<t\>$ where the $a_i$ are non-trivial elements of $A$ and each $m(i) \neq 0$ is defined to be
$2k$.

\begin{cor}\label{c3}
Any singular equation of syllable-length at most $18$ over a torsion-free group has a solution in an overgroup.
\end{cor}

This generalises \cite[Corollary 4]{BH}, which proves the same result for equations of syllable length at most $14$.

The remainder of the paper is structured as follows.  In Section \ref{cors} we prove
Corollaries \ref{c1}, \ref{c2} and \ref{c3}, assuming Theorem \ref{main}.  We then
split the proof of Theorem \ref{main} into two cases: the case where one of the
factor groups $A,B$ is cyclic is dealt with in Section \ref{cyclic}.

In our proofs we rely heavily on the theory of pictures over one-relator products and over relative presentations.  For details of the basic theory and terms used for one-relator
products the reader is referred to \cite{Ho2}; and for relative presentations see \cite{BP} or \cite[Section 3]{BEW}.  In particular \emph{aspherical} will mean aspherical 
in the sense of \cite{BP} and \emph{diagrammatically reducible} in the sense of \cite{BEW}.

\section{Proof of Corollaries}\label{cors} 

\begin{pfof}{Corollary \ref{c1}}
As mentioned in the Introduction, the proof follows immediately from Theorem \ref{main}
by setting $B:=\<t\>$.
\end{pfof}

\begin{pfof}{Corollary \ref{c2}}
For $n\in\Z$, let $A_n:=t^nAt^{-n}\subset A*\<t\>$.  The normal closure of $A$ in
$A*\<t\>$ is the free product of the $A_n$ for all $n\in\Z$.  Since $w$ has exponent sum
$0$ in $t$, it belongs to this normal closure, and hence can be uniquely written as a word
$w_0\in\ast_{n\in\Z} A_n$.

Let us denote by $\mu,M$ the minimum and maximum respectively of the sequence
of partial sums $\left(\sum_{i=1}^j m(i)\right)_{j=1}^k$.  Then it is routine to check that $w_0$ is a cyclically reduced word in $\ast_{n=\mu}^M A_n$ that contains letters from each of $A_\mu$ and $A_M$.  Write $B_-:=\ast_{n=\mu}^{M-1} A_n$, $B_+:=\ast_{n=\mu+1}^M A_n$
and 
$$H:=\frac{\ast_{n=\mu}^M A_n}{N(w_0)}.$$

Then $H=(B_-*A_M)/N(w_0)=(B_+*A_\mu)/N(w_0)$.  By hypothesis $w_0$ has a cyclically
reduced conjugate in $B_-*A_M$ of syllable-length at most $8$, so each of $B_-,A_M$ embeds
in $H$ via the natural map.  Similarly each of $B_+,A_\mu$ embeds in $H$ via the natural map.

Finally, note that $G$ can be written as an HNN extension of $H$ with stable letter $t$
and associated subgroups $B_-,B_+$.  The result follows.
\end{pfof}

\begin{pfof}{Corollary \ref{c3}}
Using the same notation as in the proof of Corollary \ref{c2}, if $w$ has syllable length
$2k\le 18$ in $A*\<t\>$, then $w_0$ has syllable length at most $k\le 9$ in
$B_-*A_M$ and in $B_+*A_\mu$.  Hence $w_0$ involves at most $4$ letters from $A_\mu$ and at most $4$ from $A_M$; equivalently, the sequence of partial sums in Corollary \ref{c2}
reaches its maximum and its minimum at most $4$ times each.  The result follows from
Corollary \ref{c2}.
\end{pfof}

\section{The Cyclic Factor Case}\label{cyclic} 

The proof of Theorem \ref{main} will be given in this and the next section.
As usual, we may reduce to the case when $A$ is generated by the $a_i$ and $B$ by the $b_i$ ($1 \leq i \leq 4$), so we assume this throughout without further comment.
Moreover if $k \leq 3$ then the result follows from \cite[Corollary 3]{BH}, so assume from now on that $k=4$.

The element $a_i$ is said to be \emph{isolated} if no $a_k$ belongs to the cyclic subgroup generated by $a_i$ for $k \neq i$; and similarly for $b_j$.  If there is an isolated pair $a_i,b_j$ for some $1 \leq i,j \leq 4$ then Theorem 1 follows from
\cite[Theorem 1]{BH}, and this fact will be used throughout what follows often without explicit comment.

In this section we prove the special case of Theorem 1 in which one of the factor groups $A$, $B$ is cyclic.  If both are cyclic Theorem 1 follows by the classical
Freiheitssatz \cite{Mag} so it can be assumed without any loss that $A$ is not cyclic and
$B = \langle t \rangle$ is infinite cyclic generated by $t$.  Therefore each $b_j$ has the form $t^{m(j)}$ for some $m(j) \in \mathbb{Z} \backslash \{ 0 \}$.
The result has been proved in \cite{Kim} except in the case where all the $m(j)$ have the same sign, so suppose without loss of generality that $m(j) > 0$ for each $j$.
The injectivity of $A \to G$ is then a result from \cite{Lev}, so it suffices to show that
$B \to G$ is injective, that is, $t$ has infinite order in $G$.

If the relative 1-relator presentation
\[
\mathcal{P} \colon G \cong \langle A,t \mid a_1 t^{m(1)} a_2 t^{m(2)} a_3 t^{m(3)}
a_4 t^{m(4)} \rangle
\]
is aspherical then $t$ is known to have infinite order in $G$ \cite{BP}, so we may assume that
$\mathcal{P}$ is not aspherical.  By \cite[Theorem 2]{BH} the result holds (that is, $t$ has infinite order in $G$) unless
\begin{equation}\label{eq1}
a_1 a_2 a_3 a_4 = 1
\end{equation}
in $A$ so we assume that this equation holds.  We separate the proof into three cases.

\medskip\noindent{\bf Case 1:} $a_1 \neq a_2 \neq a_3 \neq a_4 \neq a_1$ in $A$.  The star graph $\Gamma$ of
$\mathcal{P}$ consists of two vertices $t^{\pm 1}$, and $m(1)+m(2)+m(3)+m(4)$ edges from
$t^{-1}$ to $t$.
For a full discussion on star graphs and weight tests the reader is referred to
\cite[Section 2]{BP}, or \cite[Section 3.2]{BEW}.
Four edges are labelled $a_1,a_2,a_3,a_4$ and the remainder are labelled by $1 \in A$.  Define a weight function by assigning weight 1 to the edges labelled $1 \in A$, and weight $\frac{1}{2}$ to the other four edges.  If the four elements $a_j \in A$ are pairwise distinct then the weight function, and therefore
$\mathcal{P}$ \cite[Theorem 2.1]{BP}, is aspherical contrary to hypothesis.  Hence two of the $a_j$ are equal, and by symmetry we may assume that $a_1=a_3$.
Given this, cyclic permutation yields the symmetry $(m(1),m(2),m(3),m(4)) \leftrightarrow (m(3),m(4),m(1),m(2))$; in particular, we can work modulo $a_2 \leftrightarrow a_4$.
Furthermore, using cyclic permutation, inversion and $x \leftrightarrow x^{-1}$ we can in addition use the second symmetry $(m(1),m(2),m(3),m(4)) \leftrightarrow 
(m(2),m(1),m(4),m(3))$.  

\begin{figure}
\begin{center}
\psfig{file=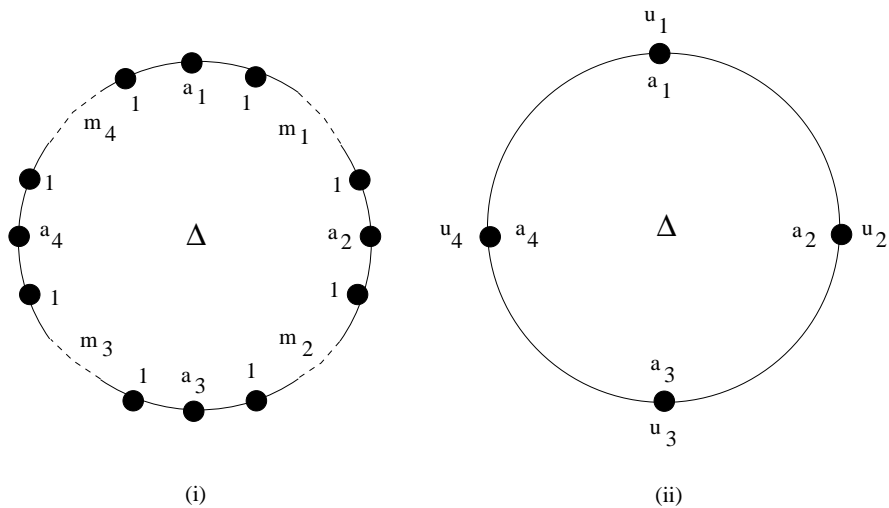}
\end{center}
\caption{the region $\Delta$}
\end{figure}

Now define a new weight function by assigning weight 1 to every edge of $\Gamma$ except for the two edges labelled $a_2$ and $a_4$, which have weight 0.  Since $\mathcal{P}$ is not aspherical there must be an admissable closed path in $\Gamma$ of weight less than 2.  Up to cyclic reordering and inversion, the label of such a path is one of:
(i) $(a_2^{-1} a_4)^m$;
(ii) $x^{-1} a_4 (a_2^{-1} a_4)^m$;
(iii) $a_2^{-1} x (a_2^{-1} a_4)^m$;
(iv) $1^{-1} a_4 (a_2^{-1} a_4)^m$; or
(v) $a_2^{-1} 1 (a_2^{-1} a_4)^m$ for some $m \geq 1$ (where $x$ denotes $a_1$ or $a_3$).
Working modulo $a_2 \leftrightarrow a_4$ it is sufficient to consider only cases (i), (ii) and (iv).

\medskip
In case (i), since $A$ is torsion-free, we have $a_2=a_4$.  Using equation (\ref{eq1}) we obtain $1 = a_1 a_2 a_3 a_4 = (a_1 a_2)^2$ in $A$ and so $a_1 a_2=1$ in $A$,
contradicting $A$ non-cyclic.  

\medskip
Consider now case (ii). If $m(1)=m(3)$ and $m(2)=m(4)$ then $G=\langle H, t | t^{m(2)} a_3 t^{m(3)} s^{-1} = 1 \rangle$, where $H= \langle A, s | s a_2 s a_4 = 1 \rangle$
is obtained from the torsion-free group $A$ by adjunction of a square root, and hence is also torsion-free. The result follows from [4, Theorem 2 (iv)]. It can be assumed then that
either $m(1) \neq m(3)$ or $m(2) \neq m(4)$.
Suppose that $\mathcal{S}$ is a non-empty reduced spherical
picture over $\mathcal{P}$. Contract the boundary of $\mathcal{S}$ to a point which is then deleted and let $\mathcal{D}$ be the dual of $\mathcal{S}$ with labelling inherited from 
$\mathcal{S}$. In particular, since the natural map from $A$ to $G$ is 
injective, each vertex label of $\mathcal{D}$  yields a word trivial in $A$. 
Then $\mathcal{D}$ is a non-empty reduced spherical diagram over
$\mathcal{P}$ whose regions are given (up to cylic permutation and inversion) by $\Delta$ of Figure 3.1(i) and so the label of each region is read in a clockwise direction whereas 
each vertex label is read anti-clockwise. For convenience we depict $\Delta$ as shown in Figure 3.1(ii). Assign angles to the corners 
of each region $\Delta$ of $\mathcal{D}$ as follows: each corner of a vertex in $\Delta$ of degree $d$ is given an angle $2\pi/d$. This way the vertices each have zero curvature; and 
if $\Delta$ has $k$ vertices $v_i$ of degree $d_i > 2$ ($1 \leq i \leq k$), that is, deg$(\Delta) = k$, then the curvature $c(\Delta)$ of $\Delta$ is given by 
\begin{equation}\label{eq2}
c(\Delta)=(2-k) \pi + 2 \pi \sum_{i=1}^k \frac{1}{d_i}
\end{equation}
which we sometimes denote by $c(d_1,\ldots,d_k)$. It follows that the sum of the curvatures of the regions of $\mathcal{D}$ is $4\pi$ (see \cite[Section 3.3]{BEW}) and it is this we 
seek to contradict. Now the 
degree of each vertex $v$ in $\mathcal{D}$ is even since the label $l(v)$ corresponds to a reduced closed path in the star graph $\Gamma$ and so if $c(\Delta) > 0$ then $d(\Delta)<4$;
and indeed the fact that $m(1)=m(3)$ and $m(2)=m(4)$ is together disallowed prevents $d(\Delta)=2$ and so forces $d(\Delta)=3$. Assume that $\mathcal{D}$ is maximal with respect to 
number of vertices of degree 2.
Then the following labels (up to cyclic permutation and inversion) for a vertex of degree 4 are disallowed: $a_1 a_3^{-1} 1 1^{-1}$; $a_1 1^{-1} 1 a_3^{-1}$; $a_1 a_3^{-1} a_1 
a_3^{-1}$; and $1 1^{-1} 1 1^{-1}$ where it is
understood that different edges are used for $1 1^{-1}$ or $1^{-1} 1$. This is because in each case there is a bridge move that would create two vertices of degree 2 but destroy at 
most one, contradicting maximality. Given this, the fact that $A=<a_1,a_4>$ is torsion-free and non-cylic, each vertex label is a word trivial in 
A and that the relation for case (ii) implies $a_1^{2}=(a_1 a_4)^{2m+1}$, an
inspection of the closed paths of length 4 in $\Gamma$ shows that if $d(v)=4$ then (up to cyclic permutation and inversion) $l(v) \in \{ a_1 a_4^{-1} a_2 a_4^{-1}, a_3 a_4^{-1} a_2 
a_4^{-1}, a_1 a_2^{-1} a_4 a_2^{-1}, a_3 a_2^{-1} a_4 a_2^{-1} \}$. Now $a_1 a_4^{-1} a_2 a_4^{-1}=1$ implies $a_1^{2}=(a_1 a_4)^{3}=(a_2 a_1)^{-3}$; and $a_1 a_2^{-1} a_4 a_2^{-1}=1$
implies $a_1^{2}=(a_1 a_2)^{3}=(a_4 a_1)^{-3}$ so these labels cannot both occur. Applying the symmetry $a_2 \leftrightarrow a_4$ it is enough to consider only $l(v) \in \{a_1 a_4^{-1} 
a_2 a_4^{-1}, a_3 a_4^{-1} a_2 a_4^{-1} \}$. Note also that, as shown in case (i), $a_2 \neq a_4$ and so if $d(v)=2$ then $l(v) \in \{ a_1 a_3^{-1}, 1 1^{-1} \}$. If $m_1 \neq m_3$ and
$m_2 \neq m_4$ it is easily shown that if $d(\Delta)=3$ then $\Delta$ has no vertices of degree 4 and so $c(\Delta) \leq 0$ (we omit the details). Applying the second symmetry 
mentioned earlier it is sufficient therefore to consider the two cases $m_1 = m_3, m_2 < m_4$ and $m_1 = m_3, m_2 > m_4$. 

In what follows, in order to deal with regions of positive curvature we use \textit{curvature distribution}.
Briefly, we locate all positive regions $\Delta$ and add $c(\Delta) > 0$ to
$c(\hat{\Delta})$ where $\hat{\Delta}$ is some suitably chosen neighbouring region.  Having done this for each $\Delta$, let $c^{\ast}(\hat{\Delta})$ denote
$c(\hat{\Delta})$ plus all possible additions of $c(\Delta)$.  If
$c^{\ast}(\hat{\Delta}) \leq 0$ for each region $\hat{\Delta}$ then the total $4 \pi$ cannot be attained which is the contradiction we require.

First assume that $m(1) = m(3)$ and $m(2) < m(4)$. If $d(u_1)=2$ in Figure 3.1(i) then $m(2) < m(4)$ forces a ($u_4$,$u_1$)-split, that is, a vertex of degree $> 2$ between $u_4$ and 
$u_1$; and 
so $c(\Delta) > 0$ implies $\Delta$ is given by Figure 3.2(i) in which  $m(2) < m(4)$ and no ($u_2$,$u_3$)-split forces $d(u_2) \geq 6$ and $m(2) < m(4)$ forces the 
($u_4$,$u_1$)-split in $\hat\Delta$ at some vertex $u$ with $d(u) \geq 6$. Distribute $c(\Delta) 
\leq c(4,6,6) = \pi/6$ to 
$c(\hat{\Delta}) \leq c(4,4,6,6) = -\pi/3$ as indicated. If $d(u_1) \geq 6$ and $c(\Delta) > 0$ then $\Delta$ is given by Figure 3.2(ii) 
\begin{figure}
\begin{center}
\psfig{file=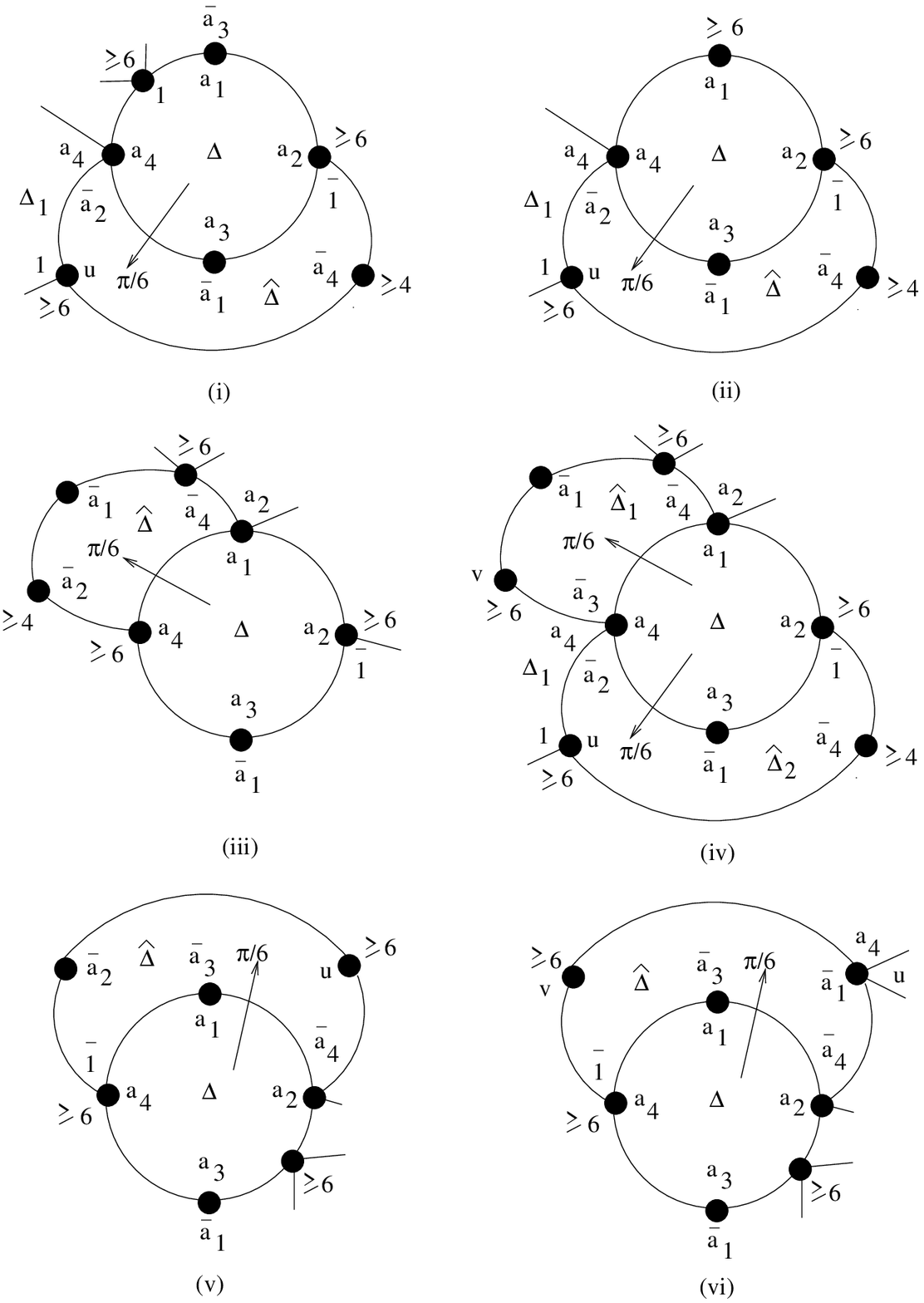}
\end{center}
\caption{positive $\Delta$ and distribution of curvature}                
\end{figure}
and again distribute $c(\Delta) \leq \pi/6$ to 
$c(\hat{\Delta}) 
\leq -\pi/3$ as shown. If $d(u_1)=4$ and $d(u_4) \geq 6$ then $c(\Delta) > 0$ implies $\Delta$ is given by Figure 3.2(iii) in which $m(2) < m(4)$ forces a split in $\hat{\Delta}$ at $u$ with
$d(u) \geq 6$. Distribute $c(\Delta) \leq \pi/6$ to $c(\hat{\Delta})\leq -\pi/3$ as shown. Finally if $d(u_1)=d(u_4)=4$ then $\Delta$ is given by Figure 3.2(iv) in which
$d(u_4)=4$ forces $d(v) \geq 6$. 
Distribute $\frac12  c(\Delta) \leq \frac12 c(4,4,6) = \pi/6$ to each of $c(\hat{\Delta}_i)\leq -\pi/3$ where $i=1,2$ as shown. This completes the distribution rules from regions of 
$\mathcal{D}$ of positive curvature. But now observe from the figures that if $\hat{\Delta}$ receives any positive curvature then $c(\hat{\Delta})\leq -\pi/3$ and $\hat{\Delta}$ 
receives at most $\pi/6$ from each of at most two neighbouring regions. It follows that $c^{\ast}(\hat{\Delta}) \leq 0$ as required.  

Now assume that $m(1) = m(3)$ and $m(2) > m(4)$. Let $d(u_1)=d(u_3)=2$ in Figure 3.1(i). Then $m(2) > m(4)$ forces a ($u_2$,$u_3$)-split and so if $c(\Delta) > 0$ then $\Delta$ is 
given by 
Figure 3.2(v)  
in which $m(2) > m(4)$ forces $d(u_4) \geq 6$. If $d(u) \geq 6$ in Figure 3.2(v) then distribute $c(\Delta) \leq \pi/6$ to $c(\hat{\Delta}_i)\leq -\pi/3$ as shown. On the other hand if 
$d(u)=4$ 
then $\Delta$ and $\hat{\Delta}$ are given by Figure 3.2(vi) in which the labelling forces $d(v) \geq 6$ and again distribute $c(\Delta) \leq \pi/6$ to $c(\hat{\Delta}_i)\leq -\pi/3$. 
If $d(u_1)>2$ and $d(u_3)=2$ in Figure 
3.1(i) then $d(\Delta)>3$.
Finally let let  $d(u_1)=2$ and $d(u_3)>2$. If any vertex other than $u_2,u_3,u_4$ has degree $> 2$ then $c(\Delta) < 0$, so assume otherwise. Then, as above,  $d(u_1)=2$ forces $d(u_4) \geq 
6$. If  $d(u_2)=4$ then $m(2) > m(4)$ forces a ($u_2$,$u_3$)-split so let $d(u_2) \geq 6$. Thus if $c(\Delta) > 0$ we must have $d(u_3)=4$.  
But then no ($u_2$,$u_3$)-split forces $m(3) \geq m(2)$ and no ($u_3$,$u_4$)-split forces $m(4) \geq m(3)$, a contradiction. This completes the 
distribution 
rules and as in the previous subcase we have $c^{\ast}(\hat{\Delta}) \leq 0$ as required.     
  
\medskip
In case (iv), the subgroup of $A$ generated by $a_2$ and $a_4$ is cyclic, generated by $a_2^{-1} a_4$.
Using equation (\ref{eq1}) we obtain $a_2^{-1} a_4 = (a_3 a_4)^2$, so the subgroup of $A$ generated by $a_3 a_4$ contains $a_2$ and $a_4$, hence also $a_3$ and $a_1$ forcing $A$ to be 
cyclic, a contradiction.

\medskip\noindent{\bf Case 2:} $a_1=a_2$ in $A$ and $b_1=b_2$ in $B$, that is, $m(1)=m(2)$. If $a_3=a_4$ then $a_1^2a_3^2=1$ by (3.1). Thus $A$ is a torsion-free homomorphic image of the Klein bottle
group, hence is locally indicable and the result then follows \cite{Ho}. So assume from now on that $a_3 \neq a_4$.

Assign weight $1$ to each edge of the star graph $\Gamma$ of $\mathcal{P}$ except for the edges labelled $a_3$ and $a_4$, which are given weight $0$. Since $\mathcal{P}$ is not 
aspherical there must be an admissable closed path in $\Gamma$ of weight less than $2$. Up to cyclic permutation and inversion the label of such a path is one of:  
(i) $(a_3^{-1} a_4)^m$;
(ii) $x^{-1} a_4 (a_3^{-1} a_4)^m$;
(iii) $a_3^{-1} x (a_3^{-1} a_4)^m$;
(iv) $1^{-1} a_4 (a_3^{-1} a_4)^m$; or
(v) $a_3^{-1} 1 (a_3^{-1} a_4)^m$ for some $m \geq 1$ where $x$ denotes $a_1$ or $a_2$.
But in case (i), since $A$ is torsion-free, we obtain $a_3=a_4$, a contradiction. The curvature arguments required for the remaining four cases are similar to those of Case 1 and so
we will omit some of the detail. 

As before, let $\mathcal{D}$ be a non-empty reduced spherical diagram over $\mathcal{P}$ whose regions are given by Figure 3.1(i)-(ii); and the same assignation of curvature to the 
corners of each region $\Delta$ is used. We once more make the assumption that $\mathcal{D}$ is maximal with respect to number of vertices of degree 2; and, subject to this, we make 
the additional assumption that the number of vertices of degree 4 with label $11^{-1}11^{-1}$ is minimal. With these assumptions  an argument using bridge moves shows that in fact we  
may disallow the vertex labels $a_1a_2^{-1}a_1a_2^{-1}$ and $11^{-1}11^{-1}$. In order to check the possible labels for a vertex of degree 4 we make use of the 
following observations which appeal, in particular, to (3.1): if $a_1a_3^{-1}a_4=1$ then $A = \langle a_3 \rangle$ is cyclic; if $a_1a_4^{-1}a_3^{-1}=1$ then $(a_3a_4)^3=1$ and it 
follows that $a_1=1$; if 
$a_1a_3a_4^{-1}=1$ then $A = \langle a_4 \rangle$ is cyclic; if $a_1a_3^{-1}a_1a_4^{-1}=1$ then $A = \langle a_1,a_3 \rangle$ where $a_3a_1a_3^{-1}=a_1^{-3}$ and so $A$ is a 
torsion-free homomorphic image of the Baumslag-Solitar group $BS(1,-3)$, hence is locally indicable and the result follows \cite{Ho}; or if $a_1a_3^{-1}a_4^{-1}=1$ then $A = \langle 
a_1,a_3 \rangle$ where $a_3a_1a_3^{-1}=a_1^{-2}$ and so $A$ is homomorphic image of the Baumslag-Solitar group $BS(1,-2)$ and similarly the result follows.
Given these observations together with $A$ torsion-free and non-cyclic it is readily verified that if $d(v)=4$ then (up to cyclic permutation and inversion) $l(v) \in 
\{a_1a_2^{-1}11^{-1}, xa_3^{-1}a_4a_3^{-1}, xa_4^{-1}a_3a_4^{-1}, xa_4^{-1}a_31^{-1}, x1^{-1}a_4a_3^{-1}, a_11^{-1}1a_2^{-1},\\ a_3a_4^{-1}a_31^{-1}, a_3a_4^{-1}1a_4^{-1}\}$ where 
$x=a_1$ or $a_2$.  

\medskip
Consider case (ii) and so $a_1^{-1}a_4(a_3^{-1}a_4)^m=1$. If $a_1a_3^{-1}a_4a_3^{-1}=1$ then $(a_3^{-1}a_4)^{m+2}=1$ and so $a_3=a_4$; if $a_1a_4^{-1}a_3=1$ then $A=\langle a_1 
\rangle$ is cyclic; if $a_1a_4a_3^{-1}=1$ then $A=\langle a_3^{-1}a_4 \rangle$ is cyclic; or if $a_3^2a_4^{-1}=1$ or $a_3a_4^{-2}=1$ then $A$ is cyclic. Therefore if $d(v)=4$ then 
$l(v) \in \{a_1a_2^{-1}11^{-1}, a_1a_4^{-1}a_3a_4^{-1} (m=1), a_2a_4^{-1}a_3a_4^{-1} (m=1), a_11^{-1}1a_2^{-1}\}$. Suppose that $m >1$. Then any vertex involving $a_3$ or $a_4$ has degree at 
least 6 and it is a routine check that $c(\Delta) \leq 0$ for each region $\Delta$. Let $m=1$. If $m(1)=m(2)=m(3)=m(4)$ then since $B=\langle t \rangle = \langle b_1,b_2,b_3,b_4 
\rangle$ it follows that $m(i)=1 (1 \leq i \leq 4)$. But then $a_1^{-1}a_4a_3^{-1}a_4=1$ implies that the subgroup of $A$ generated by loops in $\Gamma$ is cyclic (indeed generated by 
$a_1a_4^{-1}$) and the result follows; so assume otherwise. Checking then shows that if $c(\Delta)>0$ then either $m(1)=m(4)>m(3)$ and $\Delta$ is given by Figure 3.3(i)-(iii); or 
$m(1)=m(4)<m(3)$ and $\Delta$ is given by Figure 3.3(iv); or $m(3)=m(4)>m(1)$ and $\Delta$ is given by Figure 3.3(v).  
Distribute $c(\Delta) \leq \pi/6$ to $c(\hat{\Delta}) \leq -\pi/6$ as shown in Figure 3.3. If $\hat{\Delta}$ receives positive curvature across exactly one edge we can conclude 
that $c^{\ast}(\hat{\Delta}) \leq 0$. This is certainly true for the last two cases and if $m(1)=m(4)>m(3)$ again we can see it holds since in Figure 
3.3(i)-(iii) the curvature is distributed across the same $(1,a_3)$-edge.

\medskip
Consider case (iii) and so $a_3^{-1}a_1(a_3^{-1}a_4)^m=1$ and $A= \langle a_3,a_4 \rangle$. If $a_1a_4^{-1}a_3a_4^{-1}=1$ then $a_3=a_4$; if $a_1a_4^{-1}a_3=1$ or $a_1a_4a_3^{-1}=1$ 
then $A=\langle a_3^{-1}a_4 \rangle$ is cyclic; or if $a_3^2a_4^{-1}=1$ or $a_3a_4^{-2}=1$ then A is cyclic. Therefore if $d(v)=4$ then $l(v) \in \{a_1a_2^{-1}11^{-1}, 
a_1a_3^{-1}a_4a_3^{-1} (m=1), a_2a_3^{-1}a_4a_3^{-1} (m=1), a_11^{-1}1a_2^{-1}\}$. If $m>1$ then again $c(\Delta) \leq 0$ for each region $\Delta$ so let $m=1$. If 
$m(1)=m(2)=m(3)=m(4)$ then as in case (ii) the subgroup of $A$ generated by loops in $\Gamma$ is cyclic (generated 
\begin{figure}
\begin{center}
\psfig{file=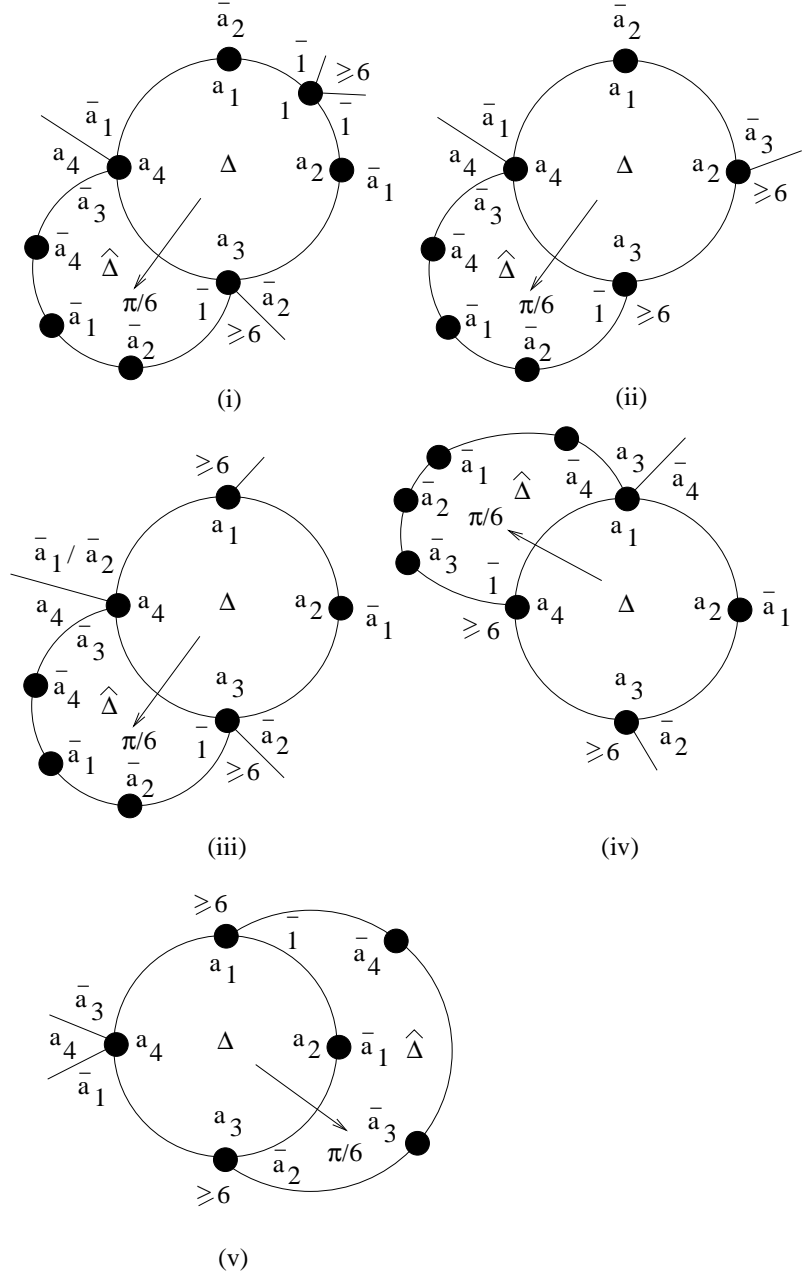}
\end{center}
\caption{positive $\Delta$ and distribution of curvature}
\end{figure}
\begin{figure}
\begin{center}
\psfig{file=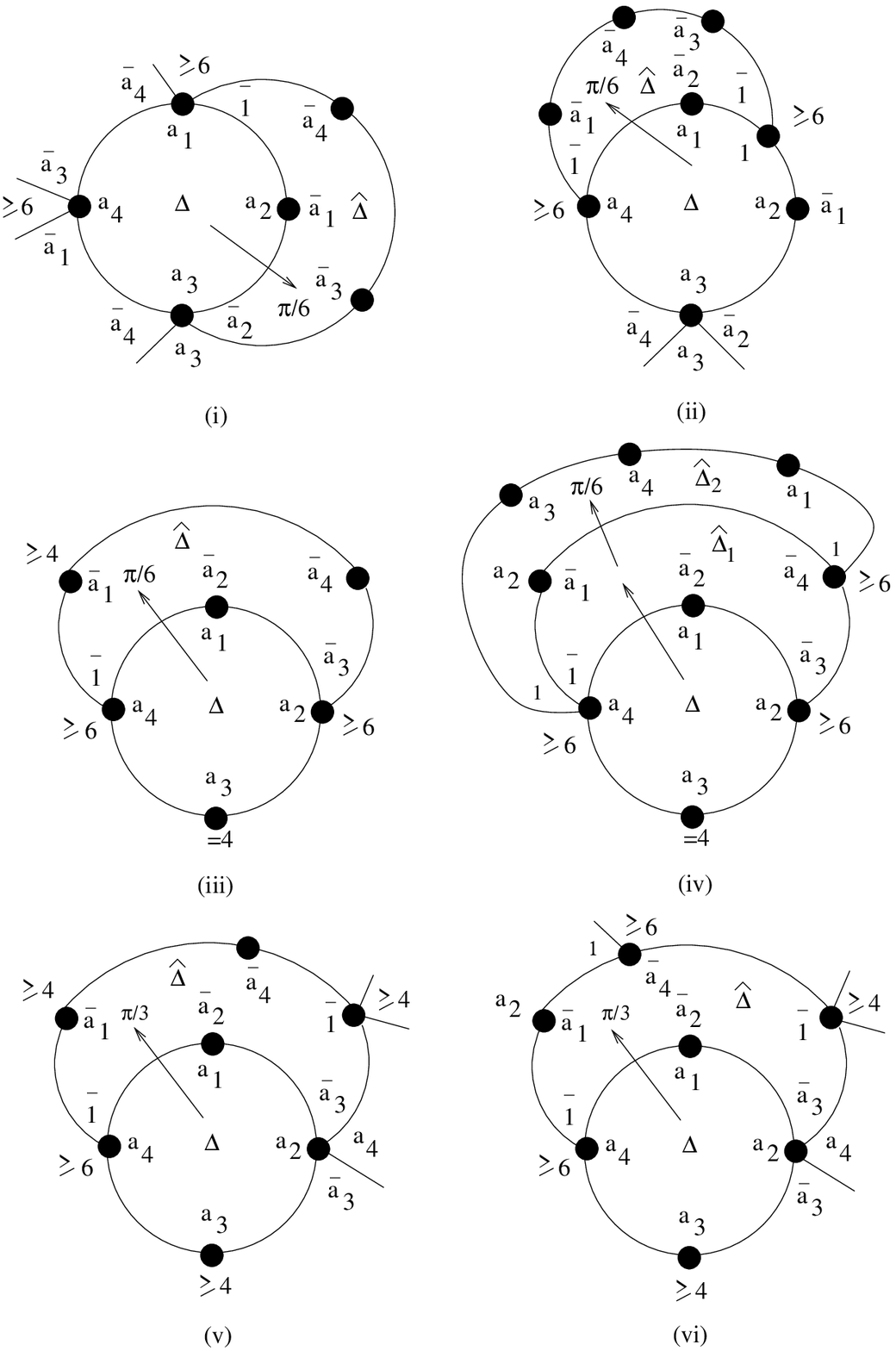}
\end{center}
\caption{positive $\Delta$ and distribution of curvature}
\end{figure}
by $a_1a_3^{-1}$), so asume otherwise. The case $m(1)=m(4)$ is 
symmetric to case (ii). If $m(1)<m(4)$ there is only one $\Delta$ for which $c(\Delta)>0$ and $\Delta$ is given by Figure 3.4(i) in which $c(\Delta) \leq \pi/6$ is added to 
$c(\hat{\Delta}) \leq -\pi/6$ as shown, so $c^{\ast}(\hat{\Delta}) \leq 0$. This leaves $m(1)>m(4)$. Checking now shows that if $c(\Delta)>0$ then $\Delta$ is given by Figure 
3.4(ii)-(vi). In Figure 3.4(ii) and (iii) $c(\Delta) \leq \pi/6$ is added to $c(\hat{\Delta}) \leq -\pi/6$ as shown; and in Figure 3.4(v) and (vi) $c(\Delta) \leq \pi/3$ is added to
$c(\hat{\Delta}) \leq -\pi/3$ as shown. In Figure 3.4(iv) however, $c(\Delta) \leq \pi/6$ is distributed to $c(\hat{\Delta}_2) \leq -\pi/3$ via $c(\hat{\Delta}_1) \leq 0$.
The key point once again is that in Figure 3.4 curvature is distributed across the same $(1,a_2)$-edge each time and it follows that $c^{\ast}(\hat{\Delta}) \leq 0$. 

\medskip
Consider case(iv). Then $a_4(a_3^{-1}a_4)^m=1$ and so, in particular, $\langle a_3,a_4 \rangle = \langle a_3^{-1}a_4 \rangle$ is cyclic. It quickly follows that if $d(v)=4$ then 
$l(v) \in \{a_1a_2^{-1}11^{-1}, a_3a_4^{-1}1a_4^{-1} (m=1), a_11^{-1}1a_2^{-1}\}$. If $m>1$ then, as before, there are no regions having positive curvature, so let $m=1$. Note also 
that if $m(1)=m(2)=m(3)=m(4)=1$ there are no vertices of degree 4 which implies $c(\Delta) \leq 0$, so we can assume otherwise. Checking then shows that if $c(\Delta)>0$ then $\Delta$ 
is given by Figure 3.5(i) or (ii) 
and in each case $c(\Delta) \leq \pi/6$ is added to $c(\hat{\Delta}) \leq -\pi/6$ as shown across the same $(1,a_2)$-edge and it follows that $c^{\ast}(\hat{\Delta}) \leq 0$.

\medskip
Finally consider case (v). Then $a_3^{-1}(a_3^{-1}a_4)^m=1$ and $\langle a_3,a_4 \rangle = \langle a_3^{-1}a_4 \rangle$ is cyclic. Therefore if $d(v)=4$ then 
$l(v) \in \{a_1a_2^{-1}11^{-1},\\ a_3a_4^{-1}a_31^{-1} (m=1), a_11^{-1}1a_2^{-1}\}$. As in case (iv) it can be assumed that $m=1$ and that $m(1)=m(2)=m(3)=m(4)=1$ does not hold. But 
checking now shows that $c(\Delta) \leq 0$ for each region $\Delta$ and we are done.  
 
\medskip\noindent{\bf Case 3:} $a_1 = a_2$ in $A$ and $b_4 \neq b_1 \neq b_2$ in $B$.  We need a few preliminary results.

\begin{lem}\label{lem3.1}
If none of the $a_j$ is isolated then $t$ has infinite order.
\end{lem}

\begin{proof}
If $a_3$ is not isolated, then either $a_1=a_2$ is a power of $a_3$, or $a_4$ is a power of $a_3$.  In the first case $A$ is cyclic by (\ref{eq1}), contrary to hypothesis.  Hence $a_4=a_3^m$ for some $m$.  Similarly $a_3=a_4^n$ for some $n$.
Hence $a_3^{mn-1}=1$ which forces $m=\pm 1$.  But if $m=-1$ then (\ref{eq1}) gives
$a_1^2=1$, a contradiction.  Hence $m=1$ and (\ref{eq1}) gives $a_1^2 a_3^2=1$.
As in Case 2 above, $A$ is then locally indicable and the result follows.
\end{proof}

\begin{lem}\label{lem3.2}
If $t$ does not have infinite order, then one of the following holds:
(i) the subgroup of $A$ generated by $a_3,a_4$ is cyclic; or (ii) $a_3^{-1} a_1$ is a power of $a_3^{-1} a_4$.
\end{lem}

\begin{proof}
As in the proof of Case 1 we construct a putative weight function on the star graph $\Gamma$ of $\mathcal{P}$.  Assign weight 0 to the two edges labelled $a_3$ and $a_4$, and weight 1 to every other edge.  Since $\mathcal{P}$ is not aspherical, there is an admissable path of weight less than 2 and the possible labels are of the form
(i) $(a_3^{-1} a_4)^m$;
(ii) $x^{-1} a_4 (a_3^{-1} a_4)^m$;
(iii) $a_3^{-1}x (a_3^{-1} a_4)^m$;
(iv) $1^{-1} a_4 (a_3^{-1} a_4)^m$; and
(v) $a_3^{-1} 1 (a_3^{-1} a_4)^m$ for some $m \in \mathbb{Z}$, where $x$ denotes $a_1$ or $a_2$.

In case (i), since $A$ is torsion-free, $a_3=a_4$.  Then $a_1^2 a_3^2=1$ by (\ref{eq1}) and as before $A$ is a homomorphic image of the Klein bottle group so is locally indicable hence the result.  The other four cases each give one of the two conclusions in the statement of the lemma.
\end{proof}

The next lemma is well-known, but we include a proof for completeness.

\begin{figure}
\begin{center}
\psfig{file=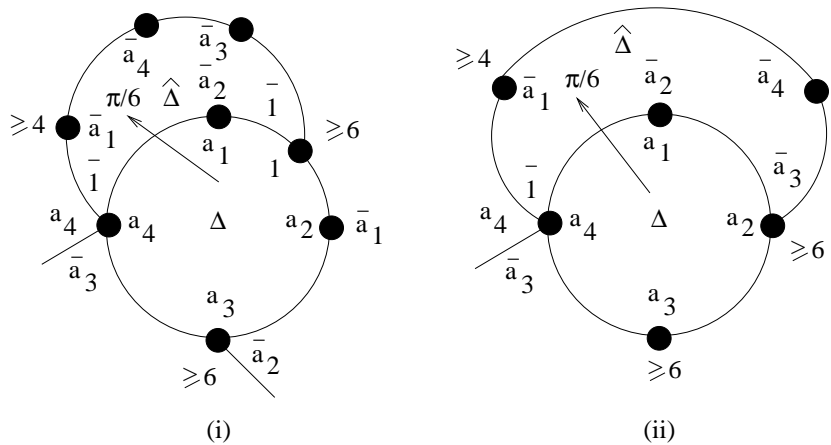}
\end{center}
\caption{positive $\Delta$ and distribution of curvature}
\end{figure}

\begin{lem}\label{lem3.3}
Let $F$ be a field, $G$ a torsion-free group, and $\alpha \in FG$ an element of the form $ag+bh$, where $g,h$ are distinct elements of $G$ and $a,b$ are non-zero elements of $F$.  Then $\alpha$ is neither a unit nor a zero-divisor in $FG$.
\end{lem}

\begin{proof}
Suppose not.  Then there exists $\beta \in FG$ with support $S$ of size $n < \infty$, such that $\alpha \beta \in F$.  The Boolean sum ($gS~\textsc{xor}~hS$) has an even
number of elements, but is contained in the support of $\alpha \beta$, which is either $\{ 1 \}$ or $\emptyset$.  Hence $gS=hS$.  For each $s \in S$, it follows that
$hg^{-1} s \in S$.  Iterating, $(hg^{-1})^n s \in S$ for all $n \in \mathbb{Z}_+$.
But $S$ is finite, so the sequence $\{ (hg^{-1} )^n s \}$ has repetitions, and
$(hg^{-1})^k=1$ for some $k > 0$.  Since $G$ is torsion-free and $g \neq h$, this is a contradiction.
\end{proof}

\begin{cor}\label{cor3.4}
If there exists a permutation $\sigma \in S_4$ such that
\[
m(\sigma(1))=m(\sigma(2)) > m(\sigma(3))=m(\sigma(4)),
\]
then $t$ has infinite order.
\end{cor}

\begin{proof}
Since the $b_i$ generate $\langle t \rangle$, we can deduce that $m(\sigma(1))$ and
$m(\sigma(3))$ are coprime.  Let $p$ be a prime factor of $m(\sigma(1))$.  Now by
\cite[Theorem 2]{BH}, if $t$ has finite order then the element
\[
\alpha := m(1) + m(2) a_1 + m(3) a_1 a_2 + m(4) a_1 a_2 a_3
\]
is a unit in $\mathbb{Q} B$.  Thus there exists $\beta \in \mathbb{Z} B$ with
setwise-coprime coefficients such that $\alpha \beta \in \mathbb{Z}$.  Reducing modulo $p$, $\alpha$ is either a unit or a zero-divisor in $Z_p B$ (depending on whether or not $p | \alpha \beta$).  But precisely two of the coefficients of $\alpha$ are coprime to $p$, so this contradicts Lemma \ref{lem3.3}.
\end{proof}
Now let us return to the proof of the theorem in Case 3 (in which it is more convenient to work with pictures over one-relator products).

Since $\mathcal{P}$ is not aspherical, there is a non-empty reduced spherical
picture $S$ over the one-relator product
$\langle A \ast B \mid a_1 b_1 a_2 b_2 a_3 b_3 a_4 b_4 \rangle$ \cite{Ho2}.
The vertices of $S$ have label $a_1b_1a_2b_2a_3b_3a_4b_4$ (up to cyclic permutation and inversion) and the regions of $S$ are either $A$-regions or $B$-regions whose label equals 1 in 
$A$ or $B$ except possibly for the distinguished region $\Delta_0$ if it is a $B$-region.
This time assign angles to the corners of $S$ as follows: each corner of an $n$-gon is assgned an angle $(n-2) \pi /n$.  This way the curvature of each region is 0 and, since the total 
curvature is $4 \pi$, there exists at least one vertex, $v$ say, having positive curvature, that is, the sum of the incident angles is less than $2 \pi$.  (See  \cite[Section 3]{Ho2}.)  

Assume until otherwise stated that none of the regions incident at $v$ is $\Delta_0$.  Since there are eight edges incident at $v$ some of the incident regions must be 2-gons.  Indeed if there are at most two 2-gons then clearly $v$ does not have positive curvature, so assume otherwise.  Observe that no 2-gons at $v$ can share an edge for
otherwise one of these is an $A$-region and by Lemma \ref{lem3.1} and $A$ non-cyclic must have label $(a_1 a_2^{-1})^{\pm 1}$.  Hence one of the labels of the adjacent 2-gonal $B$-region is $b_1^{\pm 1}$, while the other label is either $b_2^{\mp 1}$ or $b_4^{\mp 1}$ which contradicts $b_4 \neq b_1 \neq b_2$.  We may also assume that no four of the corners at $v$ belong to 2-gons.  For otherwise, since no two of the 2-gons share an edge, all four 2-gons belong to the same free factor $A$ or $B$.  But
$a_3$ and $a_4$ cannot be labels of 2-gonal regions, so our four 2-gons are all $B$-regions.  But $b_4 \neq b_1 \neq b_2$ then forces the $b_j$ to be equal in pairs, and the result follows by Corollary \ref{cor3.4}.

It follows that $v$ must have exactly three 2-gonal corners, four 3-gonal corners and the eighth corner belonging to a 3-, 4- or 5-gon.  (For example, if there are two 4-gonal corners then $v$ has curvature at most $2 \pi - [3(\pi/3)+2(\pi/4)]=0$.)

Now the corner at $v$ labelled $b_1$ cannot be 2-gonal.  For otherwise, by hypothesis, it yields $b_1=b_3$.  Since the $b_j$ cannot be equal in pairs, we then have
$b_2 \neq b_4$.  But also $b_4 \neq b_1 \neq b_2$, so the corners at $v$ labelled $b_2,b_4$ are not 2-gonal.  Moreover, the corners labelled $a_1,a_2$ are then not 2-gonal since they are adjacent to the corner labelled $b_1$, while those labelled $a_3$ and $a_4$ are not 2-gonal since neither $a_3$ nor $a_4$ is involved in a relation of length 2 among the $a_j$.  But all of this contradicts the fact that $v$ has three incident 2-gonal corners.  Observe also that at most one of the corners labelled $b_4,a_1$ can be 2-gonal, and the same holds for the corners labelled $a_2,b_2$.  Hence the three 2-gonal corners at $v$ consist of: $b_3$; one of $b_4,a_1$; and one of $a_2,b_2$.

At least one $A$-region incident at $v$ is 3-gonal, so there is a relator of length 3 among the $a_j$.  Suppose first of all that such a relator involves $a_1$ or $a_2$.  Letting $x$ denote $a_1 = a_2$, we have one of
(i) $x^3$,
(ii) $x^2 x^{-1}$,
(iii) $x^2 a_3^{\pm 1}$,
(iv) $x^2 a_4^{\pm 1}$,
(v) $xa_3^{\pm 2}$,
(vi) $x a_4^{\pm 2}$,
(vii) $x^{\pm 1} a_3 a_4$,
(viii) $x^{\pm 1} a_4 a_3$,
(ix) $xa_3 a_4^{-1}$,
(x) $xa_3^{-1} a_4$,
(xi) $x^{-1} a_3^{-1} a_4$, or
(xii) $x^{-1} a_3 a_4^{-1}$.
Any one of the relators (i)-(vii), together with the relator $x^2 a_3 a_4$ from (\ref{eq1}) and the fact that $A$ is torsion-free, implies that $A$ is cyclic, contrary to 
\begin{figure}
\begin{center}
\psfig{file=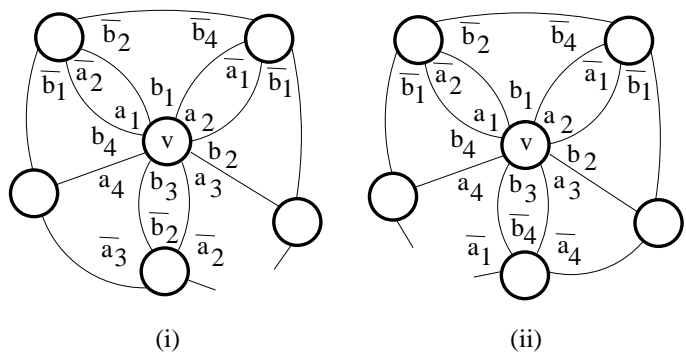}
\end{center}
\caption{neighbourhood of vertex $v$}
\end{figure}
hypothesis.  Relator (viii) can be rewritten as $a_3 = a_4^{-1} x^{\pm 1}$ so
$1=x^2 a_3 a_4 = x^2 a_4^{-1} x^{\pm 1} a_4$ and $A$ is the homomorphic image of the
soluble Baumslag--Solitar group BS$(1,\pm 2)$, so locally indicable.
Relator (ix) gives $1 = x^2 a_3 a_4 = x^2 a_3 xa_3$, so $A$ is cyclic and a similar
contradiction is obtained from (x).  Relator (xi) can be rewritten as
$x=a_3^{-1} a_4$.  By Lemma \ref{lem3.2} either the subgroup of $A$ generated by $a_3,a_4$ is cyclic, or $a_3^{-1} x$ is a power of $a_3^{-1} a_4 = x$ and in either case $A$ is then cyclic.  A similar argument applies to (xii), noting that the second possibility in Lemma \ref{lem3.2} can be conjugated to: $xa_3^{-1}$ is a power of $a_4 a_3^{-1} (=x)$.
Thus no 3-gonal relation among the $a_j$ involves $a_1$ or $a_2$.

Now consider the 2-gon with label $b_3$ at $v$.  We cannot have $b_1=b_3$ by the argument given above, so the other label of this 2-gon is either $b_2^{-1}$ or
$b_4^{-1}$.  It follows that one of the neighbouring regions has $a_1^{-1}$ or
$a_2^{-1}$ corner label, so cannot be 3-gonal.  Thus every other corner of $v$ is 2-gonal or 3-gonal.  In particular, the $a_1$- and $a_2$-corners are 2-gonal and the
$b_1$-, $b_2$- and $b_4$-corners are 3-gonal.  Now the other label of the 2-gon at the $a_1$-corner of $v$ is $a_2^{-1}$, and vice versa.  Hence the other two corners of the 3-gon at 
the $b_1$-corner are labelled $b_2^{-1}$ and $b_4^{-1}$.  Hence we have a relation $m(1)=m(2)+m(4)$ in $B$.  The two cases for the vertex $v$ are shown in Figure 3.6(i), (ii).

Now recall from Lemma \ref{lem3.1} that one of the $a_j$ is isolated in $A$ and so we may assume that none of the $b_j$ is isolated in $B$.  If the $m(j)$ consist of two equal pairs, then the result follows from Corollary \ref{cor3.4}.
There are only two other possibilities: either three of the $m(j)$ are equal, and the fourth divides them (and hence is equal to 1 since the $m(j)$ are setwise coprime), or two of the $m(j)$ are equal, and each of the other two divide them (and are coprime to one another).  In the first case the only possibility is that
$m(2)=m(3)=m(4) > 1$ and $m(1)=1$, in which case $m(1)=m(2)+m(4)$ is clearly false.
In the second case $m(3)$ is equal to either $m(2)$ or $m(4)$, and $m(1)$ divides $m(3)$.  But then $m(2)+m(4) > m(3) > m(1)$, so the equation $m(1)=m(2)+m(4)$ again fails.

In conclusion, $v$ does not have positive curvature.

To complete the proof suppose that $\Delta_0$ is incident at $v$ and that $\Delta_0$ is a $B$-region of degree $k_0$.  It follows from the above that there are at most three 2-gons distinct from $\Delta_0$ incident at $v$ and so $v$ has curvature at most
$2 \pi - [ 4(\pi/3) + (k_0-2) \pi / k_0]$ which implies that the total curvature of $S$ is at most $(2-k_0 /3) \pi < 4 \pi$, a contradiction and Theorem 1 follows for this case.


\section{The Case of No Cyclic Factors}\label{noncyclic}

In this section we prove Theorem 1 under the assumption that neither $A$ nor $B$ is cyclic.  The statement of the theorem will hold if the relative presentation
\[
\mathcal{P}_X \colon \langle A \ast B, X \mid a_1 Xb_1 X^{-1} a_2 Xb_2 X^{-1} a_3 Xb_3 X^{-1} a_4 Xb_4 X^{-1} \rangle
\]
is aspherical \cite{Kim}.
The star graph $\Gamma_X$ of $\mathcal{P}_X$ consists of two disjoint bouquets
of circles, with 4 circles in each.  One of these corresponds to $A$ and has edge labels
$a_1,a_2,a_3,a_4$; and the other to $B$ and has edge labels $b_1,b_2,b_3,b_4$.

Suppose by way of contradiction that $\S$ is a reduced spherical picture over
$\mathcal{P}_X$. As in Case 1(ii) of Section 3   
contract the boundary of $\S$ to a point which is then deleted and let $\mathcal{D}$ denote the dual, whose labelling is inherited from $\S$.  The regions of $\mathcal{D}$ are 
$\Delta^{\pm 1}$ where $\Delta$ is given by Figure 4.1(i).  The vertices of $\mathcal{D}$ are either $A$-vertices or $B$-vertices whose label equals 1 in $A$ or $B$ except possibly for the label of the distinguished ($A$ or $B$) vertex $v_0$ of degree $k_0$.

The region $\Delta$ is called \textit{inner} if $v_0$ is not a vertex of $\Delta$, otherwise $\Delta$ is a \textit{boundary} region.  The degree of $\Delta$, denoted $d(\Delta)$, is defined to be the number of vertices of $\Delta$ of degree $>2$ except possibly $v_0$ if $\Delta$ is a boundary region.  
Give each corner of $\mathcal{D}$ at a vertex of degree $d$ the angle $2 \pi / d$ as before.  Therefore the curvature of a region 
$\Delta$ of $\mathcal{D}$ is again given by (\ref{eq2}) and the total curvature of the regions is $4\pi$. 

\begin{figure}
\begin{center}
\psfig{file=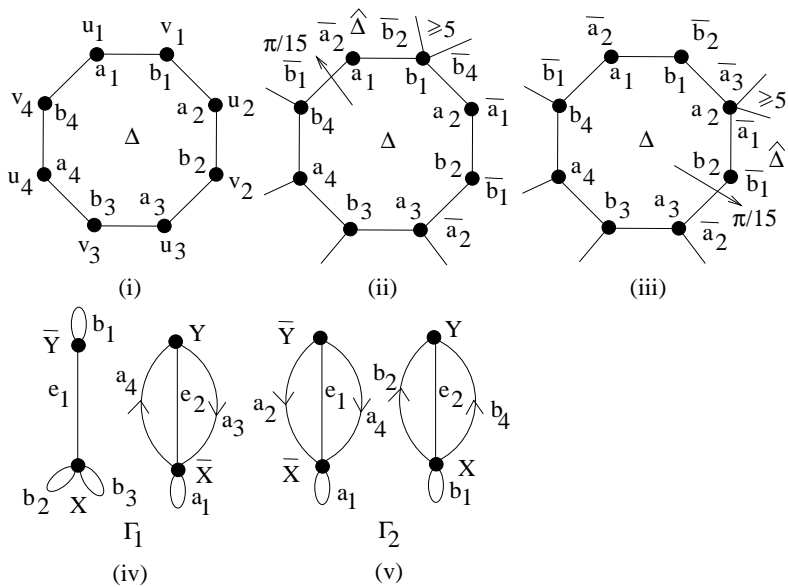}
\end{center}
\caption{region $\Delta$, distribution of curvature and star graphs}
\end{figure}

\noindent If $c(\Delta) \leq 0$ and
$d(\Delta) \geq 4$ for each inner region $\Delta$ of $\mathcal{D}$ then the total curvature is at most $\sum c(\Delta ')$ where the sum is taken over all the boundary regions $\Delta '$ of $\mathcal{D}$.  But then
$c(\Delta') \leq c(k_0,3,3,3)= \frac{2 \pi}{k_0} < \frac{4 \pi}{k_0}$, so the total curvature is less than $4 \pi$ from which we conclude that $\mathcal{P}_X$ is aspherical. We  
use curvature distribution as described in Section 3.   
Again let $c^{\ast}(\hat{\Delta})$ denote
$c(\hat{\Delta})$ plus all possible additions of $c(\Delta)$ according to the distribution rules.  If 
$c^{\ast}(\hat{\Delta}) \leq 0$ for each inner $\hat{\Delta}$ and 
$c^{\ast}(\hat{\Delta}) < 4 \pi / k_0$ for each boundary region $\hat{\Delta}$ then the total $4 \pi$ cannot be attained and $\mathcal{P}_X$ is aspherical.

Put
\[
S_A = \{ a_1 a_2^{\pm 1}, a_1 a_3^{\pm 1}, a_1 a_4^{\pm 1}, a_2 a_3^{\pm 1}, a_2 a_4^{\pm 1}, a_3 a_4^{\pm 1} \} \quad \text{and}
\]
\[
S_B = \{ b_1 b_2^{\pm 1}, b_1 b_3^{\pm 1}, b_1 b_4^{\pm 1}, b_2 b_3^{\pm 1},
b_2b_4^{\pm 1}, b_3b_4^{\pm 1} \}.
\]
It can be assumed without any loss that $n_B \leq n_A$ where $n_A$, $n_B$ denotes the number of admissable paths in $\Gamma_X$ contained in $S_A$, $S_B$ respectively.
If $n_A > 3$ then $A$ is cyclic (a contradiction); or if $n_A=n_B=0$ or $n_A=1$, $n_B=0$ then $d(\Delta) \geq 6$, and so $c(\Delta) \leq 0$,
for each inner region of $\mathcal{D}$
therefore $\mathcal{P}_X$ is aspherical by the above comments, so assume otherwise.  Given this, up to symmetry (obtained from cyclic permutation and inversion of the relator), the cases to be considered are the following.
\[
\begin{array}{ll}
n_A=1 \colon &a_1=a_2;\, a_1=a_2^{-1};\, a_1=a_3; \, a_1=a_3^{-1}.\\
n_A=2 \colon &a_1=a_2,\, a_3=a_4; \, a_1 = a_2,\, a_3=a_4^{-1}; \, a_1=a_2^{-1},\,
a_3=a_4^{-1};\\
&a_1=a_3, \, a_2=a_4; \, a_1=a_3, \, a_2=a_4^{-1}; \, a_1 = a_3^{-1}, \, a_2=a_4^{-1}.\\
n_A=3 \colon &a_1=a_2=a_3; \, a_1=a_2^{-1}=a_3; \, a_1=a_2=a_3^{-1}.
\end{array}
\]
The following assumption will be made throughout.

\medskip

({\bf{A}}) The number of $A$ vertices $v \neq v_0$ of $\mathcal{D}$ of degree 2 is maximal.

\medskip

First consider
$n_A=n_B=1$.  Up to symmetry the subcases to be considered are:
$a_1=a_2^{\pm 1}$, $b_1=b_2^{\pm 1}$;
$a_1=a_2^{\pm 1}$, $b_1=b_3^{\pm 1}$;
$a_1=a_2^{\pm 1}$, $b_2=b_3^{\pm 1}$;
$a_1=a_2^{\pm 1}$, $b_2=b_4^{\pm 1}$; and
$a_1=a_3^{\pm 1}$, $b_1=b_3^{\pm 1}$.
Checking vertex labels shows that $d(\Delta) \geq 6$ when
$a_1=a_2$, $b_1=b_2^{-1}$;
$a_1=a_2$, $b_2=b_4^{\pm 1}$;
$a_1=a_2^{-1}$, $b_1=b_2^{\pm 1}$;
$a_1=a_2^{-1}$, $b_2=b_4$;
$a_1=a_3$, $b_1=b_3^{-1}$; or
$a_1=a_3^{-1}$, $b_1=b_3^{\pm 1}$
and so we are left with eleven subcases.

Throughout the following $\Delta$ will be an inner region unless stated otherwise.

\medskip
Let $a_1=a_2$ and $b_1=b_2$.  If $c(\Delta)>0$ then $\Delta$ is given by Figure 4.1(ii), (iii).
Add $c(\Delta) \leq c(3,3,3,3,5) = \pi / 15$ to $c(\hat{\Delta})$ as indicated.
Observe that $\hat{\Delta}$ receives $\pi / 15$ across the $(a_2^{-1},b_1^{-1})$-edge each time, that $d(\hat{\Delta}) \geq 6$ and that $\hat{\Delta}$ contains a vertex of degree $\geq 5$.  It follows that $c^{\ast}(\hat{\Delta}) \leq c(3,3,3,3,3,5) + \pi/15 = - \pi / 5$ if $\hat{\Delta}$ is inner.  If $\tilde{\Delta}$ is a boundary region then either $c(\tilde{\Delta}) \leq c(k_0,3,3,3,3) = 2 \pi/k_0 - \pi/3$ or
$c^{\ast} (\tilde{\Delta}) \leq c(k_0,3,3,3,3,3) + \pi /15 = 2 \pi/k_0 - 3 \pi/5$,
so $c(\tilde{\Delta}),c^{\ast}(\tilde{\Delta}) < 4 \pi/k_0$ and the result follows.

\medskip
Let $a_1=a_2^{-1}$ and $b_2=b_4^{-1}$.  Then the relative presentation $\mathcal{P}_X$ is equivalent to the relative presentation
\[
\mathcal{P}_{X,1} \colon \langle A \ast B,X,Y \mid Y^{-1}X^{-1}b_2^{-1}Xa_1 X^{-1},
Yb_1 Y^{-1} a_3 X^{-1} b_3 Xa_4 \rangle
\]
whose star graph $\Gamma_1$ is given by Figure 4.1(iv) in which the labels $e_1=e_2=1$.
(We will use $e$, $e_1$ or $e_2$ to denote 1 in $A$ or $B$.)
Assign the weight $\frac{1}{2}$ to all the edges of $\Gamma_1$.  Then we obtain an aspherical weight function unless at least one of $b_2^2 b_3^{\pm 1}=1$,
$b_3^2 b_2^{\pm 1}=1$, $a_4 a_3 a_1^{\pm 1}=1$ holds.  If
$a_4 a_3 a_1^{\pm 1}=1$, $b_2^2 b_3^{\pm 1} \neq 1$, $b_3^2 b_2^{\pm 1}\neq 1$ then assign weight 0 to (the edge labelled -- for ease of presentation we will often identify an edge with its label) $e_2$, 1 to $a_1$ and $\frac{1}{2}$ to all other edges; if $a_4 a_3 a_1^{\pm 1} \neq 1$, $b_2^2 b_3^{\pm 1}=1$,
$b_3^2 b_2^{\pm 1} \neq 1$ then assign 0 to $b_1$, 1 to $b_3$ and $\frac{1}{2}$ to all other edges; if $a_4 a_3 a_1^{\pm 1} \neq 1$, $b_2^2 b_3^{\pm 1} \neq 1$,
$b_3^2 b_2^{\pm 1}=1$ then assign 0 to $e_1$, 1 to $b_2$ and $\frac{1}{2}$ to all other edges; if $a_4 a_3 a_1^{\pm 1}=1$, $b_2^2 b_3^{\pm 1}=1$ then assign 0 to $b_1$ and $e_2$, 1 to $b_3$ and $a_1$, and $\frac{1}{2}$ to all other edges; or if
$a_4 a_3 a_1^{\pm 1}=1$, $b_3^2 b_2^{\pm 1}=1$ then assign 0 to $e_1$ and $e_2$, 1 to $b_2$ and $a_1$ and $\frac{1}{2}$ to all other edges.  The fact that $A$ and $B$ are non-cyclic ensures that each of these weight functions is aspherical and the result follows.  (For the reader's benefit we note here that if $\phi$ is one of the weight functions defined above then in each case one confirms that

\begin{figure}
\begin{center}
\psfig{file=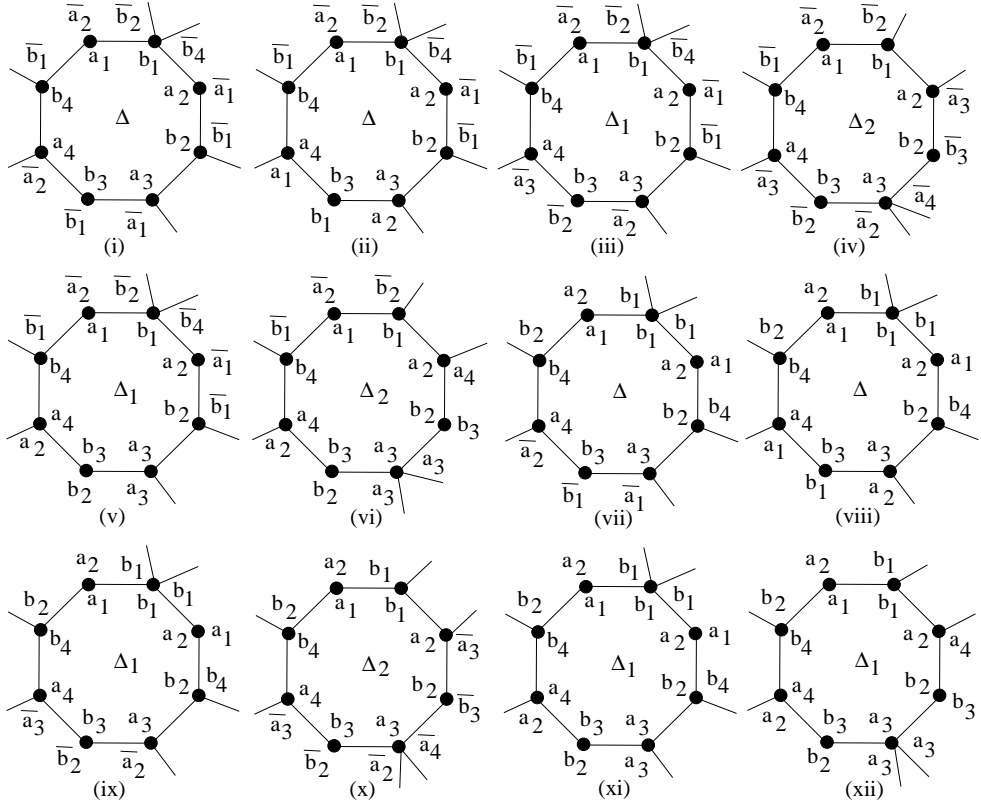}
\end{center}
\caption{regions $\Delta$ such that $d(\Delta) <6$}
\end{figure}

\noindent $(1-\phi(e_2))+(1-\phi(b_2))+(1-\phi(a_1))+(1-\phi(e_1)) \geq 2$ and that
$(1-\phi(b_1))+(1-\phi(a_3))+(1-\phi(b_3))+(1-\phi(a_4))\geq 2$.  Calculations such as these are made implicitly throughout what follows.)

\medskip
Let $a_1=a_3$ and $b_1=b_3$.  Then the presentation $\mathcal{P}_X$ is equivalent to the relative presentation
\[
\mathcal{P}_{X,2} \colon \langle A \ast B,X,Y \mid Y^{-1} Xa_1 X^{-1} b_1 X,
Ya_2 X^{-1} b_2 Ya_4 X^{-1} b_4 \rangle
\]
whose star graph $\Gamma_2$ is given by Figure 4.1(v).  Assigning weight $\frac{1}{2}$ to all edges of $\Gamma_2$ defines an aspherical weight function unless either $a_1^{\pm 1} (a_2^{-1} a_4)^{\pm 1}=1$, in which case $a_4$ is isolated; or $b_1^{\pm 1}(b_2 b_4^{-1})^{\pm 1}=1$, in which case $b_4$ is isolated.  If $a_1^{\pm 1} (a_2^{-1} a_4)^{\pm 1}=1$ then assign weight 0 to $e_1$, 1 to $a_1$ and $\frac{1}{2}$ to all other edges; or if
$b_1^{\pm 1} (b_2 b_4^{-1})=1$ then assign 0 to $e_2$, 1 to $b_1$ and $\frac{1}{2}$ to all other edges.  The fact that $A$ and $B$ are non-cyclic ensures that both weight functions are aspherical.

The regions $\Delta$ such that $d(\Delta) <6$ for the remaining 8 subcases are given by Figure 4.2(i)-(xii).  In Figure 4.2(i), if $d(v_1)=d(v_2)=d(u_3)=3$ then
$l(v_2)=b_1^{-1} b_2 b_4$ 
\begin{figure}
\begin{center}
\psfig{file=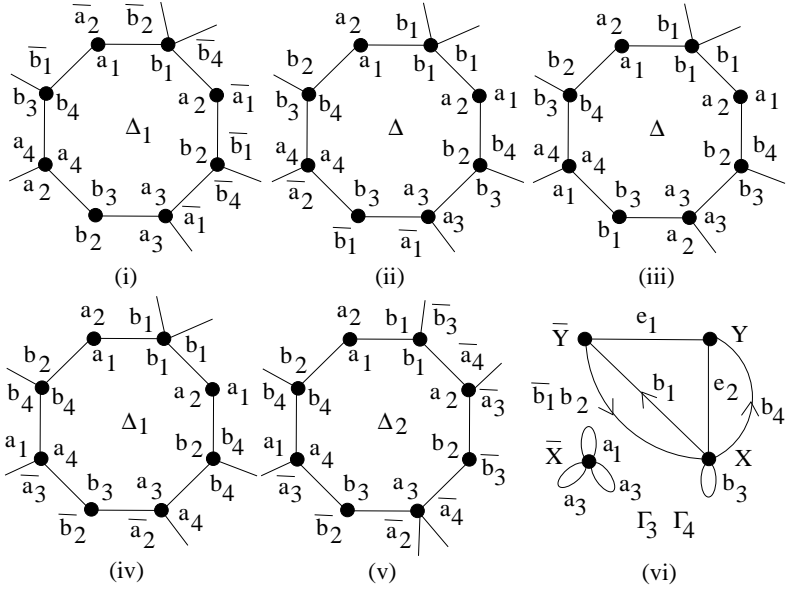}
\end{center}
\caption{regions with four degree 3 vertices and star graphs}
\end{figure}
and $l(u_3)=a_3 a_1^{-1} a_4$ forcing the isolated pair
$a_4$ and $b_4$; or if
$d(v_1)=d(u_4)=d(v_4)=3$ then $l(u_4)=a_2^{-1} a_4 a_3$ and $l(v_4)=b_4 b_1^{-1} b_2$ again forcing $a_4,b_4$ isolated, so it can be assumed that $d(v_1) > 3$.
In Figure 4.2(ii), if $d(v_1)=d(v_2)=d(u_3)=3$ then $l(v_2)=b_1^{-1} b_2 b_4$ and $l(u_3)=a_3 a_2 a_4$ forcing $a_4,b_4$ isolated; or if $d(v_1)=d(u_4)=d(v_4)=3$ then
$l(u_4)=a_1 a_4 a_3$ and $l(v_4)=b_4 b_1^{-1} b_2$ forcing $a_4,b_4$ isolated, so let
$d(v_1) > 3$.  In Figure 4.2(iii), if $d(v_1)=d(v_2)=d(u_3)=3$ then
$l(v_2)=b_1^{-1} b_2 b_4$ and $l(u_3)=a_3 a_2^{-1} a_4$ forcing $a_4,b_4$ isolated; or if $d(v_1)=d(u_4)=d(v_4)=3$ then $l(v_4)=b_4 b_1^{-1} b_3$ and
$l(u_4)=a_3^{-1} a_4 a_4$ forcing $a_3,b_4$ isolated, so let $d(v_1) > 3$.
In Figure 4.2(iv), if $d(u_3)=d(v_1)=d(u_2)=3$ then $l(u_2)=a_2 a_3^{-1} a_4$ and
$l(v_1)=b_2^{-1} b_1 b_4$ forcing $a_4,b_4$ isolated; or if $d(u_3)=d(u_4)=d(v_4)=3$ then $l(u_4)=a_3^{-1} a_4 a_1$ and $l(v_4)=b_4 b_1^{-1} b_4$ forcing $a_4,b_1$ isolated, so let $d(u_3)>3$.  In Figure 4.2(v), if $d(v_1)=d(v_2)=d(u_3) =3$ then
$l(v_2)=b_1^{-1} b_2 b_4$ and $l(u_3)=a_3 a_3 a_4$ forcing $a_4,b_4$ isolated; or if
$d(v_1)=d(u_4)=d(v_4)=3$ then either $l(v_4)=b_4 b_1^{-1} b_2$ and
$l(u_4)=a_2 a_4 a_3$ or $l(v_4)=b_4 b_1^{-1} b_3^{-1}$ and
$l(u_4)=a_2 a_4 a_3^{-1}$ forcing $a_4,b_4$ isolated each time, so let $d(v_1)>3$.
In Figure 4.2(vi), $d(u_3) > 3$.  In Figure 4.2(vii)-(ix), $d(u_1)>3$.
In Figure 4.2(x), if $d(u_3)=d(v_1)=d(u_2)=3$ then
$l(u_2)=a_2 a_3^{-1} a_4$ and $l(v_1)=b_1 b_1 b_4$ forcing $a_4,b_4$ isolated; or if
$d(u_3)=d(u_4)=d(v_4)=3$ then either $l(u_4)=a_3^{-1} a_4 a_2$ and $l(v_4)=b_4 b_2 b_1$ or $l(u_4)=a_3^{-1} a_4 a_1^{-1}$ and $l(v_4)=b_4 b_2 b_1^{-1}$ forcing $a_4,b_4$ isolated each time, so let $d(u_3) > 3$.  In Figure 4.2(xi), $d(v_1) > 3$; and in (xii),
$d(u_3) > 3$.  In each of the above Figures 4.2(i)-(xii), the assumption that the remaining four vertices of degree $>2$ each has degree 3 forces either $A$ or $B$ to be cyclic or an isolated pair except for the five regions shown in Figure 4.3
for each of which there is an aspherical weight function on $\Gamma_X$ as follows.

If $\Delta$ is $\Delta_1$ of Figure 4.3(i) or is $\Delta$ of (iii) then $b_4$ is isolated so in each case assign weight 1 in $\Gamma_X$ to each of
$a_1$, $a_2$, $b_1$, $b_2$ and $b_3$, assign 0 to $b_4$ and assign $\frac{1}{2}$ to each of $a_3$ and $a_4$; 

\begin{figure}
\begin{center}
\psfig{file=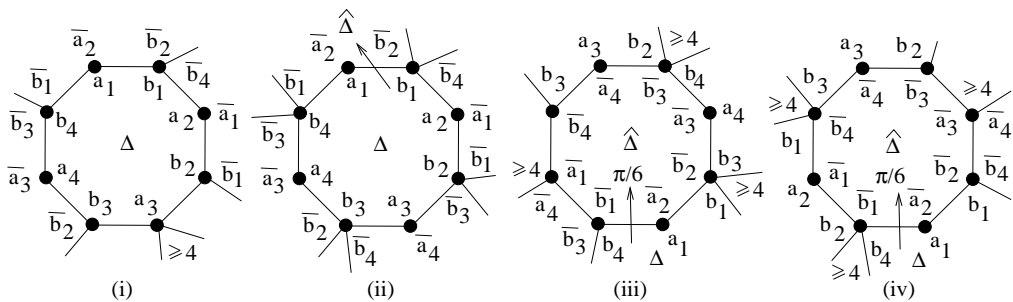}
\end{center}
\caption{positive regions and curvature distribution}
\end{figure}

if $\Delta$ is $\Delta$ of (ii) then $b_2$ is isolated so assign weight 1 to $a_1$, $a_2$, $b_1$, $b_3$ and $b_4$, assign 0 to $b_2$ and assign $\frac{1}{2}$ to each of $a_3$ and $a_4$; or if $\Delta$ is $\Delta_1$ of (iv) or
$\Delta_2$ of (v) then $a_4$ is isolated so assign the weight 1 to
$a_1$, $a_2$, $a_3$, $b_2$ and $b_3$, assign 0 to $a_4$ and assign $\frac{1}{2}$ to each of $b_1$ and $b_4$.

In conclusion we can assume that $c(\Delta) \leq 0$ for each inner region $\Delta$.
But from Figure 4.2 we conclude also that $d(\Delta) \geq 5$ for any region and so $\mathcal{P}_X$ is aspherical.


\medskip
We consider now the six cases for $n_A=2$.

\medskip
First let $a_1=a_2$ and $a_3=a_4$.  Then $\mathcal{P}_X$ is equivalent to
\[
\mathcal{P}_{X,3} \colon \langle A \ast B,X,Y \mid Y^{-1} Xa_1 X^{-1} b_1,
Y^2 b_1^{-1} b_2 Xa_3 X^{-1} b_3 Xa_3 X^{-1} b_4 \rangle
\]
whose star graph $\Gamma_3$ is given by Figure 4.3(vi).  Assume that at least one of
$b_1=b_2$, $b_2=b_3$, $b_3=b_4$ or $b_4=b_1$ holds and so by symmetry we can take
$b_1=b_2$.  Assigning in $\Gamma_3$ weight 1 to both the $a_3$ edges and to $e_1$, the weight $\frac{1}{2}$ to $b_1^{-1} b_2$, $b_1$, $e_2$ and $b_4$, and the weight 0 to $a_1$ and $b_3$ yields an aspherical weight function ($\theta$, say) unless
$b_1=b_2 \in \langle b_3 \rangle$ or $b_4 \in \langle b_3 \rangle$.  Suppose that
$b_1=b_2 \in \langle b_3 \rangle$.  Then assign weight 1 to both the $a_3$ edges and $e_1$, weight $\frac{1}{2}$ to $b_1^{-1} b_2$, $b_1$, $e_2$ and $b_3$, and weight 0 to
$a_1$ and $b_4$.  Any admissable path of weight less than 2 forces $A$ or $B$ cyclic or
$b_1 b_3^{\pm 1}=1$ which implies $n_B \geq 3$, a contradiction, so the weight function is aspherical.  Now suppose that $b_4 \in \langle b_3 \rangle$.  Assigning weight 1 to both $a_3$ edges and $e_2$, weight $\frac{1}{2}$ to $b_1^{-1} b_2$, $e_1$, $b_4$ and $b_3$, and weight 0 to $b_1$ and $a_1$ yields an aspherical weight function except when $b_4 = b_1^2$, so assume this holds.  Then, in particular, $b_1 \neq b_4^{\pm 1}$,
$b_1 \neq b_3^{\pm 1}$, $b_2 \neq b_3^{\pm 1}$, $b_2 \neq b_4^{\pm 1}$, $b_3 \neq b_4^{\pm 1}$ and
$b_2 b_3^{-1} \notin \{ b_i^{\pm 1} \colon 1 \leq i \leq 4 \}$.  Moreover, $a_1=a_2$ and $a_3=a_4$ implies $d(u_i) \neq 3$ ($1 \leq i \leq 4$).  Any attempt at labelling now shows that $d(\Delta) \geq 4$ and $c(\Delta) \leq 0$ for any inner region $\Delta$ and the result follows.

\medskip
Now assume that $b_1 \neq b_2$, $b_2 \neq b_3$, $b_3 \neq b_4$ and $b_4 \neq b_1$.
The same weight function $\theta$ as defined above again forces either
$b_2 \in \langle b_3 \rangle$ or $b_4 \in \langle b_3 \rangle$.  Let $b_2 \in \langle b_3 \rangle$ and note that this is symmetric to the case $b_4 \in \langle b_3 \rangle$. 

\begin{figure}
\begin{center}
\psfig{file=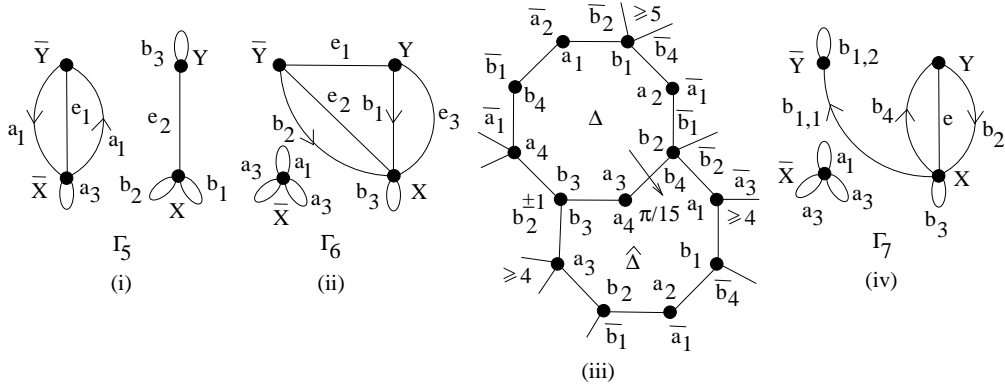}
\end{center}
\caption{star graphs and curvature distribution}
\end{figure}

Checking shows that $d(\Delta) \geq 4$ for each region of $\mathcal{D}$ and so if
$c(\Delta) \leq 0$ for each inner $\Delta$ the result follows.  In fact if
$c(\Delta) > 0$ then $\Delta$ is given by Figure 4.4(i) or (ii).  Let $\Delta$ be as in Figure 4.4(i).  Then $b_4 b_1^{-1} b_3^{-1} = b_1 b_4^{-1} b_2^{-1}=1$ and so
$b_2 = b_3^{-1}$.  Given this, assigning weight 1 in $\Gamma_3$ to both the $a_3$ edges, $b_1^{-1} b_2$, $b_1$ and $b_4$, and weight 0 to the remaining edges yields an aspherical 
weight function.  Let $\Delta$ be as in Figure 4.4(ii).  Then $c(\Delta)>0$ forces at least one of $b_4 b_1^{-1} b_3^{-1}=1$, $b_1 b_4^{-1} b_2^{-1}=1$,
$b_2 b_3^{-1} b_1^{-1}=1$ or $b_3 b_2^{-1} b_4^{-1}=1$.  If
$b_4 b_1^{-1} b_3^{-1}=b_1 b_4^{-1} b_2^{-1}=1$ we are back in the previous case and any other pair forces $B$ cyclic.
Since $b_4 b_1^{-1} b_3^{-1}=1$, $b_1 b_4^{-1} b_2^{-1}=1$ is symmetric with
$b_3 b_2^{-1} b_1^{-1} =1$, $b_3 b_2^{-1} b_4^{-1}=1$ (respectively), we consider only the first two subcases.
Consider first $b_4 b_1^{-1} b_3^{-1}=1$.  Then add
$c(\Delta) \leq c(3,4,4,4) = \pi /6$ to $c(\hat{\Delta})$ as shown in Figure 4.4(iii) where it is assumed that $\hat{\Delta}$ is inner and that $d(u_3)=d(u_4)=2$ in
$\hat{\Delta}$.  If $d(v_3)=3$ in $\hat{\Delta}$ then $B$ is cyclic so $c(\hat{\Delta}) \leq c(3,3,4,4,4)= - \pi /6$.  On the other hand if at least one of $d(u_3)$, $d(u_4)$ does not 
equal 2 then (as noted above) it must be at least 4 and then again
$c(\hat{\Delta}) \leq c(3,3,4,4,4)$.  Thus if $\hat{\Delta}$ is inner then
$c^{\ast} (\hat{\Delta}) \leq 0$, otherwise $c^{\ast}(\hat{\Delta}) \leq c(k_0,3,3,4,4) < 4 \pi / k_0$.  A similar argument applies to $\hat{\Delta}$ of Figure 4.4(iv) for the subcase 
$b_1 b_4^{-1} b_2^{-1}=1$.

\medskip
Let $a_1=a_2$ and $a_3=a_4^{-1}$.  Then $\mathcal{P}_X$ is equivalent to
\[
\mathcal{P}_{X,4} \colon \langle A \ast B,X,Y | Y^{-1} Xa_1 X^{-1} b_1,
Y^2 b_1^{-1} b_2 X a_3 X^{-1} b_3 X a_3^{-1} X^{-1} b_4 \rangle
\]
whose star graph $\Gamma_4$ is again given by Figure 4.3(vi).  Assigning in $\Gamma_4$ weight 1 to both the $a_3$ edges and $e_1$, weight $\frac{1}{2}$ to $b_1^{-1} b_2$, $b_1$, $e_2$ 
and $b_4$, and weight 0 to $a_1$ and $b_3$ yields an aspherical weight function unless $b_2 \in \langle b_3 \rangle$ or $b_4 \in \langle b_3 \rangle$.  If
$b_1=b_2 b_4$ and $b_2 \in \langle b_3 \rangle$ then assigning weight 1 to both the $a_3$ edges, $b_1^{-1} b_2$, $b_1$ and $b_4$, and weight 0 to the remaining edges yields an aspherical weight function; or if $b_1=b_2 b_4$ and $b_4 \in \langle b_3 \rangle$ then an aspherical weight function is obtained by assigning weight 1 to both the $a_3$ edges, $b_1^{-1} b_2$, $e_2$ and $b_4$, and weight 0 to the remaining edges.  It can be assumed therefore that $b_1 \neq b_2 b_4$.

\medskip
Let $b_2 = b_4^{-1}$.  Then $\mathcal{P}_X$ is equivalent to
\[
\mathcal{P}_{X,5} \colon \langle A \ast B,X,Y \mid Y^{-1} Xa_3^{-1} X^{-1} b_2^{-1} X,
Ya_1 X^{-1} b_1 Xa_1 Y^{-1} b_3 \rangle
\]
whose star graph $\Gamma_5$ is given by Figure 4.5(i).  If $b_1 \notin \langle b_2 \rangle$ then an aspherical weight function is obtained by assigning in $\Gamma_5$ weight 1 to edges $e_2$ and $b_2$, weight $\frac{1}{2}$ to both the $a_1$ edges, $e_1$ and $a_3$, and weight 0 to $b_1$ and $b_3$; or if $b_1 \in \langle b_2 \rangle$ then assigning weight 1 to $b_1$, weight 0 to $b_3$ and weight $\frac{1}{2}$ to the remaining edges yields an aspherical weight function on noting that $b_1=b_2^{\pm 1}$ would imply $n_B \geq 3$.  It can be assumed then that $b_2 \neq b_4^{-1}$.

\medskip
We are left to consider when $b_1 \neq b_2 b_4$, $b_2 \neq b_4^{-1}$ and either
$b_4 \in \langle b_3 \rangle$ or $b_2 \in \langle b_3 \rangle$.  If $b_1=b_2$ and
$b_1=b_4$ then $n_B \geq 3$; or if $b_1 \neq b_2$ and $b_1 \neq b_4$ then checking the possible labels shows that $d(\Delta) \geq 4$ and $c(\Delta) \leq 0$ for each inner region $\Delta$ and the result follows.  (We remark that assumption (\textbf{A}) is used here.)  Since $b_1=b_2$, $b_1 \neq b_4$ is symmetric to $b_1 \neq b_2$, $b_1=b_4$ we consider only the latter case and then $\mathcal{P}_X$ is equivalent to
\[
\mathcal{P}_{X,6} \colon \langle A \ast B,X,Y \mid Y^{-1} b_1 Xa_1 X^{-1} ,
Y^2 b_2 Xa_3 X^{-1} b_3 Xa_3^{-1} X^{-1} \rangle
\]
whose star graph $\Gamma_6$ is given by Figure 4.5(ii).  If $b_4 \in \langle b_3 \rangle$ then assigning in $\Gamma_6$ weight 1 to both $a_3$ edges and $e_1$, weight
$\frac{1}{2}$ to $e_2$, $b_1$, $e_3$ and $b_3$, and weight 0 to $a_1$ yields an aspherical weight function on noting that $b_1=b_3^{\pm 1}$ implies $n_B \geq 3$; so let
$b_2 \in \langle b_3 \rangle$.  Assigning weight 1 to both $a_3$ edges and $e_2$, weight $\frac{1}{2}$ to $e_1$, $b_2$, $e_3$ and $b_3$, and weight 0 to $b_1$ yields an aspherical weight function unless $b_2=b_1^2$, so assume this holds.  Then
$b_2 \neq b_3^{\pm 1}$ for otherwise $B$ is cyclic; $b_2 \neq b_1^{\pm 1}$ and
$b_2 \neq b_4^{\pm 1}$ for otherwise $n_B \geq 3$; and $l(v)=b_2b_4 w$ forces
$d(v) \geq 4$.
All of this implies $d(\Delta) \geq 4$ and $c(\Delta) \leq 0$ unless $\Delta$ is given by Figure 4.5(iii) in which $d(v_1) \geq 5$.  Add $c(\Delta) \leq c(3,4,4,5) = \pi / 15$ to $c(\hat{\Delta})$ as shown.  If $d(u_2) \geq 4$ in $\hat{\Delta}$ then, assuming $\hat{\Delta}$ inner, $c^{\ast} (\hat{\Delta}) \leq c(3,4,4,4,4) + \pi /15 =
- 4 \pi / 15$; or if $d(u_2)=2$ as shown then
$c^{\ast}(\hat{\Delta}) \leq c(3,3,4,4,4) + \pi/15 = - \pi/10$.
On the other hand if $\hat{\Delta}$ is a boundary region then
$c^{\ast}(\hat{\Delta}) \leq c(k_0,3,3,4,4) + \pi/15 < \frac{4 \pi}{k_0}$ and the result follows.

\medskip
Let $a_1=a_2^{-1}$ and $a_3=a_4^{-1}$.  Then $\mathcal{P}_X$ is equivalent to
\[
\mathcal{P}_{X,7} \colon \langle A \ast B,X,Y \mid Y^{-1} Xa_1 X^{-1} b_1,
Yb_1 Y^{-1} b_2 Xa_3 X^{-1} b_3 Xa_3^{-1} X^{-1} b_4 \rangle
\]
whose star graph $\Gamma_7$ is given by Figure 4.5(iv) in which the labels
$b_{1,1} = b_{1,2} = b_1$.  Assigning in $\Gamma_7$ weight 1 to $b_3$ and both the $a_3$ edges, weight $\frac{1}{2}$ to $e$, $b_{1,1}$, $b_2$ and $b_4$, and weight 0 to $a_1$ and $b_{1,2}$ yields an aspherical weight function unless $b_2b_4=1$, so assume this holds.  Then $\mathcal{P}_X$ is equivalent to
\[
\mathcal{P}_{X,8} \colon \langle A \ast B,X,Y \mid Y^{-1}Xa_1^{-1} X^{-1} b_2 Xa_3 X^{-1}, b_1 Yb_3 Y^{-1} \rangle
\]

\begin{figure}
\begin{center}
\psfig{file=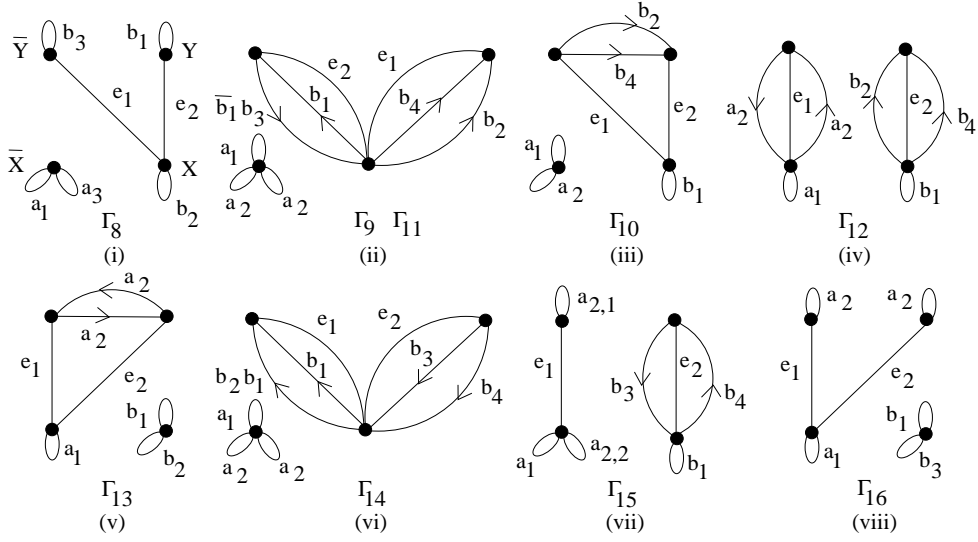}
\end{center}
\caption{star graphs}
\end{figure}

\noindent whose star graph $\Gamma_8$ is given by Figure 4.6(i).  Assigning in $\Gamma_8$ weight 1 to $e_1$, $e_2$ and $a_3$ and weight 0 to all other edges yields an aspherical weight 
function.

\medskip
Let $a_1=a_3$ and $a_2=a_4$.  Then $\mathcal{P}_X$ is equivalent to
\[
\mathcal{P}_{X,9} \colon \langle A \ast B,X,Y \mid Y^{-1} Xa_1 X^{-1} b_1,
YXa_2 X^{-1} b_2 Yb_1^{-1} b_3 Xa_2 X^{-1} b_4 \rangle
\]
whose star graph $\Gamma_9$ is given by Figure 4.6(ii).  Assigning in $\Gamma_9$ weight 1 to both the $a_2$ edges, weight 0 to $a_1$ and weight $\frac{1}{2}$ to the remaining edges yields an aspherical weight function unless $b_1=b_3$ or $b_2=b_4$.  If $b_1=b_3$ and $b_2=b_4$ then $\mathcal{P}_X$ is aspherical because
$\langle A \ast B,X \mid Xa_1 X^{-1} b_1 Xa_2 X^{-1} b_2 \rangle$ is aspherical
\cite[Theorem 2.3]{BP}.
The case $b_2=b_4$ is symmetric to $b_1=b_3$ so it is enough to consider $b_1=b_3$, in which case $\mathcal{P}_X$ is equivalent to
\[
\mathcal{P}_{X,10} \colon \langle A \ast B,X,Y \mid Y^{-1} Xa_1 X^{-1} b_1 Xa_2 X^{-1}, Yb_2 Yb_4 \rangle
\]
whose star graph $\Gamma_{10}$ is given by Figure 4.6(iii).  Assigning in $\Gamma_{10}$ weight 1 to $e_1$, $e_2$ and $a_1$ and weight 0 to the remaining edges yields an aspherical weight function.

\medskip
Let $a_1=a_3$ and $a_2=a_4^{-1}$.  Then $\mathcal{P}_X$ is equivalent to
\[
\mathcal{P}_{X,11} \colon \langle A \ast B,X,Y \mid Y^{-1} Xa_1 X^{-1} b_1,
YXa_2 X^{-1} b_2 Yb_1^{-1} b_3 Xa_2^{-1} X^{-1} b_4 \rangle
\]
whose star graph $\Gamma_{11}$ is also given by Figure 4.6(ii).  Assigning in
$\Gamma_{11}$ weight 1 to both the $a_2$ edges, weight 0 to the $b_1$ edge and weight $\frac{1}{2}$ to the remaining edges yields an aspherical weight function unless
$b_1=b_3$ or $b_2=b_4$.  Now $b_2=b_4$ is symmetric to $b_1=b_3$ so is enough to consider $b_1=b_3$ in which case $\mathcal{P}_X$ is equivalent to
\[
\mathcal{P}_{X,12} \colon \langle A \ast B,X,Y \mid Y^{-1} Xa_1 X^{-1} b_1 X,
Ya_2 X^{-1} b_2 Ya_2^{-1} X^{-1} b_4 \rangle
\]
whose star graph $\Gamma_{12}$ is given by Figure 4.6(iv).  Assigning in $\Gamma_{12}$ weight $\frac{1}{2}$ to each edge yields an aspherical weight function unless
$b_1b_2^{\pm 1}=1$, $b_1b_4^{\pm 1}=1$, $b_2b_4^{-1}=1$ or $b_1^{\pm 1} b_2b_4^{-1}=1$.
But $b_1b_2^{\pm1}=1$ or $b_1b_4^{\pm 1}=1$ implies $n_B \geq 3$, so let $b_2b_4^{-1}=1$ in which case $\mathcal{P}_X$ is equivalent to
\[
\mathcal{P}_{X,13} \colon \langle A \ast B,X,Y \mid Y^{-1} X^{-1} b_2 Xa_1 X^{-1} b_1X, Ya_2 Ya_2^{-1} \rangle
\]
whose star graph $\Gamma_{13}$ is given in Figure 4.6(v).  Assigning in $\Gamma_{13}$ weight 1 to $e_1$, $e_2$ and $b_1$ and weight 0 to the remaining edges yields an aspherical weight function.  This leaves $b_1^{\pm 1} b_2 b_4^{-1}=1$ in which case assign in $\Gamma_{12}$ weight 1 to $b_1$, weight 0 to $e_2$ and weight $\frac{1}{2}$ to the remaining edges.  A routine check shows that any admissable path of weight less than 2 involving the edges $b_1$, $b_2$, $b_4$ or $e_2$ forces $B$ cyclic and it follows that this yields an aspherical weight function.

\medskip
Let $a_1=a_3^{-1}$ and $a_2=a_4^{-1}$.  Then $\mathcal{P}_X$ is equivalent to
\[
\mathcal{P}_{X,14} \colon \langle A \ast B,X,Y \mid Y^{-1} Xa_1 X^{-1} b_1,
YXa_2 X^{-1} b_2 b_1 Y^{-1} b_3 Xa_2 X^{-1} b_4 \rangle
\]
whose star graph $\Gamma_{14}$ is given by Figure 4.6(vi).  Assigning in $\Gamma_{14}$ weight 1 to both the $a_2$ edges, weight 0 to $a_1$ and weight $\frac{1}{2}$ to the remaining edges yields an aspherical weight function unless $b_1b_2=1$ or $b_3b_4=1$.  Now $b_3b_4=1$ is symmetric to $b_1b_2=1$ so it is enough to consider $b_1b_2=1$ in which case $\mathcal{P}_X$ is equivalent to
\[
\mathcal{P}_{X,15} \colon \langle A \ast B,X,Y \mid Y^{-1} Xa_1 X^{-1} b_1X,
Ya_2 Y^{-1} b_3 Xa_2^{-1} X^{-1} b_4 \rangle
\]
whose star graph $\Gamma_{15}$ is given by Figure 4.6(vii) in which
$a_{2,1}=a_{2,2}=a_2$.  Assigning in $\Gamma_{15}$ weight 1 to $a_{2,2}$ and $e_1$, weight 0 to $a_1$ and $a_{2,1}$, and weight $\frac{1}{2}$ to the remaining edges yields an aspherical weight function unless $b_1 b_3^{\pm 1} =1$, $b_1b_4^{\pm 1}=1$, $b_3b_4=1$ or
$b_1^{\pm 1} b_4 b_3=1$.  But $b_1b_3^{\pm 1}=1$ or $b_1b_4^{\pm 1}=1$ implies
$n_B \geq 3$, so let $b_3b_4=1$ in which case $\mathcal{P}_X$ is equivalent to
\[
\mathcal{P}_{X,16} \colon \langle A \ast B,X,Y \mid Y^{-1} X^{-1} b_3^{-1} Xa_1 X^{-1} b_1 X, Ya_2 Y^{-1} a_2^{-1} \rangle
\]
whose star graph $\Gamma_{16}$ is given by Figure 4.6(viii).  Assigning in $\Gamma_{16}$ weight 1 to
$e_1$, $e_2$ and $b_3$ and weight 0 to the remaining edges yields an aspherical weight function.  This leaves $b_1^{\pm 1} b_4 b_3=1$ in which case assigning in $\Gamma_{15}$ weight 1 to the edges $e_1$, $a_{2,2}$ and $b_1$, weight $\frac{1}{2}$ to $b_3$ and $b_4$, and weight 0 to $e_2$, $a_{2,1}$ and $a_1$ yields an aspherical weight function.

\medskip
We turn now to the case $n_A=3$.

\medskip
Let $a_1=a_2=a_3$, in which case $a_4$ is then isolated.  Then $\mathcal{P}_X$ is equivalent to
\[
\mathcal{P}_{X,17} \colon \langle A \ast B,X,Y \mid Y^{-1} Xa_1 X^{-1} b_1,
Y^2 b_1^{-1} b_2 Yb_1^{-1} b_3 Xa_4 X^{-1} b_4 \rangle
\]
whose star graph $\Gamma_{17}$ is given by Figure 4.7(i).  Assigning in $\Gamma_{17}$ weight 1 to $b_4$, weight 0 to $e_2$ and weight $\frac{1}{2}$ to the remaining six edges gives an aspherical weight function unless either $b_1=b_2$ or $b_1=b_3$ or $b_2=b_3$.  If $b_1=b_2$ then assigning weight 1 to $a_1$,
$b_1^{-1} b_2$ and $e_1$, weight $\frac{1}{2}$ to $b_1^{-1} b_3$ and $b_4$ and weight 0 to $a_4$, $b_1$ and $e_2$ gives an aspherical weight function except 
\begin{figure}
\begin{center}
\psfig{file=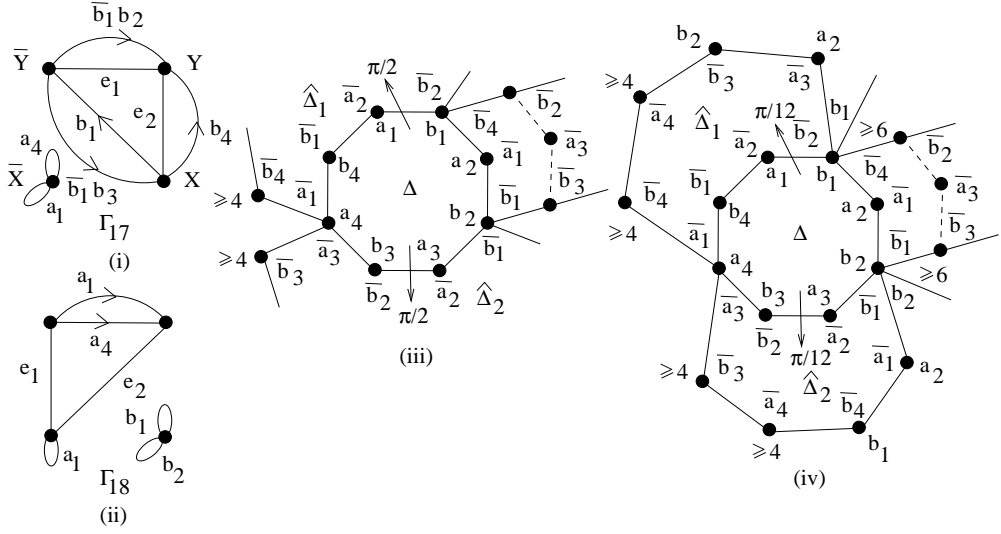}
\end{center}
\caption{star graphs and curvature distribution}
\end{figure}
when $b_3=b_4^{\pm 1}$ since all other paths in the $b_i$ of weight less than 2 forces $b_3$ or $b_4$ isolated; if $b_1=b_3$ then assigning weight 1 to $a_4$ and $e_2$, 
weight 0 to 
$a_1$ and $b_1$ and weight $\frac{1}{2}$ to the remaining four edges similarly yields an aspherical weight function unless $b_2=b_4$; or if $b_2=b_3$ then assigning weight 1 to $b_1^{-1} b_3$, weight 0 to $b_1$ and weight $\frac{1}{2}$ to the remaining six edges yields an aspherical weight function unless $b_1=b_4$.

\medskip
Let $b_1=b_2$ and $b_3=b_4^{\pm 1}$.  Then $d(u_i) \neq 3$ ($1 \leq i \leq 4$),
$d(u_4) \geq 4$ and $d(v_i) \neq 3$ ($1 \leq i \leq 4$).  Moreover $d(u_1)=2$ forces $l(v_4)=b_4b_1^{-1} w$ or $b_4 b_2^{-1} w$ and so $d(v_4) \geq 4$; $d(u_2)=2$ forces either $d(v_1)=b_1 b_4^{-1} w$ and $d(v_1) \geq 4$ or
$l(v_2)=b_3^{-1} b_2 w$ and $d(v_2) \geq 4$; and $d(u_3)=2$ forces
$l(v_3)=b_1^{-1} b_3 w$ or $b_2^{-1} b_3 w$ and $d(v_3) \geq 4$.  It follows that $d(\Delta) \geq 4$ and $c(\Delta) \leq 0$ for each interior region $\Delta$ hence the result.

\medskip
Let $b_1=b_3$ and $b_2=b_4$.  Then $\mathcal{P}_X$ is equivalent to
\[
\mathcal{P}_{X,18} \colon \langle A \ast B,X,Y \mid Y^{-1} X^{-1} b_2 X a_1 X^{-1} b_1 X_1, Ya_1 Ya_4 \rangle
\]
whose star graph $\Gamma_{18}$ is given by Figure 4.7(ii).  Assigning in $\Gamma_{18}$ weight 1 to $e_1$, $e_2$ and $b_1$ and weight 0 to the remaining four edges gives an aspherical weight function.

\medskip
This leaves $b_1=b_4$ and $b_2=b_3$.  Using $d(u_i)=3$ and $d(v_i) \neq 3$ ($1 \leq i \leq 4$) together with assumption (\textbf{A}) in particular, we find that if $c(\Delta) > 0$ then $\Delta$ is given by Figure 4.7(iii).  Note that there are two possible regions $\Delta$ according to $l(u_2)=a_2 a_1^{-1}$ or $a_2 a_3^{-1}$ as shown.  Note also that assumption (\textbf{A}) forces $d(v_4) \geq 4$ in
$\hat{\Delta}_1$ and $d(v_3) \geq 4$ in $\hat{\Delta}_2$.
Assume until otherwise stated that $\hat{\Delta}_1$ and $\hat{\Delta}_2$ are interior regions in Figure 4.7(iii).
If either $d(u_3) \geq 4$ or $d(v_3) \geq 4$ in $\hat{\Delta}_1$ then
\begin{figure}
\begin{center}
\psfig{file=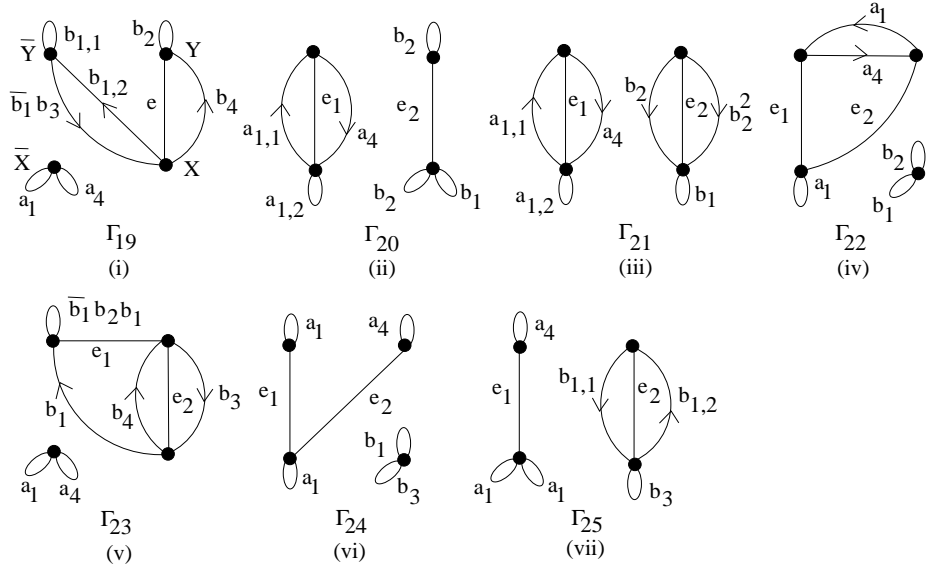}
\end{center}
\caption{star graphs}
\end{figure}
$c(\hat{\Delta}_1) \leq c(4,4,4,4,4)=- \frac{\pi}{2}$ so add
$c(\Delta) \leq c(4,4,4) = \frac{\pi}{2}$ to $\hat{\Delta}_1$ across the
$(a_2^{-1},b_2^{-1})$-edge as indicated; or if either $d(u_1) \geq 4$ or
$d(v_4) \geq 4$ in $\hat{\Delta}_2$ then add $c(\Delta)$ to $c(\hat{\Delta}_2) \leq - \frac{\pi}{2}$ across the $(a_2^{-1},b_2^{-1})$-edge of $\hat{\Delta}_2$ as indicated.  Assume otherwise so that $\hat{\Delta}_1$ and $\hat{\Delta}_2$ are given by Figure 4.7(iv).
Then $d(v_1) \geq 6$ and $d(v_2) \geq 6$ in $\Delta$ so add
$\frac{1}{2} c(\Delta) \leq \frac{1}{2} c(4,6,6) = \frac{\pi}{12}$ to each of
$c(\hat{\Delta}_i) \leq c(4,4,4,6) = - \frac{\pi}{6}$ as shown.  We then have
$c^{\ast}(\hat{\Delta}) \leq 0$ if $\hat{\Delta}$ is interior or if $\hat{\Delta}$ is a boundary region either $c^{\ast}(\hat{\Delta})=c(\hat{\Delta})=c(k_0,4,4)$ or
$c^{\ast}(\hat{\Delta}) \leq c(k_0,4,4,4,4) + \frac{\pi}{2}$ or
$c^{\ast}(\hat{\Delta}) \leq c(k_0,4,4,4) + \frac{\pi}{12}$, in which case
$c^{\ast}(\hat{\Delta}) < \frac{4 \pi}{k_0}$ and the result follows.

\medskip
Let $a_1=a_2^{-1}=a_3$ in which case $a_4$ is isolated.  Then $\mathcal{P}_X$ is equivalent to
\[
\mathcal{P}_{X,19} \colon \langle A \ast B,X,Y \mid Y^{-1}Xa_1 X^{-1} b_1,
Yb_1 Y^{-1} b_2 Yb_1^{-1} b_3 Xa_4 X^{-1} b_4 \rangle
\]
whose star graph $\Gamma_{19}$ is given by Figure 4.8(i) in which $b_{1,1}=b_{1,2}=b_1$.
Assigning in $\Gamma_{19}$ weight 1 to $a_4$, weight 0 to $a_1$ and weight $\frac{1}{2}$ to the remaining edges yields an aspherical weight function unless
$b_4 b_2^{\pm 1} =1$ or $b_3 b_1^{\pm 1}=1$.  Since these cases are symmetric we consider only $b_3 b_1^{\pm 1}=1$.  Assigning weight 1 to $a_4$ and $b_{1,2}$, weight 0 to $a_1$ and $e_1$ and weight $\frac{1}{2}$ to the remaining edges yields an aspherical weight function unless $b_4 b_2^{\pm 1}=1$.

\medskip
 Let
$b_1=b_3^{-1}$ and $b_4 b_2^{\pm 1}=1$.  Then $\mathcal{P}_X$ is equivalent to
\[
\mathcal{P}_{X,20} \colon \langle A \ast B,X,Y \mid Y^{-1}X^{-1}b_1 Xa_1^{-1} X^1,
Yb_2 Y^{-1} a_4 X^{-1} b_2^{\pm 1} Xa_1 \rangle
\]
whose star graph $\Gamma_{20}$ is given by Figure 4.8(ii) in which
$a_{1,1}=a_{1,2}=a_1$.  Assigning in $\Gamma_{20}$ weight 1 to $e_1$, $a_{1,2}$ and both the $b_2$ edges, and weight 0 to the remaining edges yields an aspherical weight function.

\medskip
  Let $b_1=b_3$ and $b_4b_2=1$.  Then $\mathcal{P}_X$ is equivalent to
\[
\mathcal{P}_{X,21} \colon \langle A \ast B,X,Y \mid Y^{-1} b_2 Xa_1 X^{-1} b_1 X,
b_2^{-2} Ya_1^{-1} X^{-1} Y a_4 X^{-1} \rangle
\]
whose star graph $\Gamma_{21}$ is given by Figure 4.8(iii) in which $a_{1,1}=a_{1,2}=a_1$.
Assigning in $\Gamma_{21}$ weight 1 to $e_1$, $a_{1,2}$, $e_2$ and $b_2^2$ and weight 0 to the remaining edges yields an aspherical weight function. 

\medskip
 This leaves $b_1=b_3$ and $b_2=b_4$. Then $\mathcal{P}_X$ is equivalent to
\[
\mathcal{P}_{X,22} \colon \langle A \ast B,X,Y \mid Y^{-1} X^{-1} b_2 Xa_1 X^{-1} b_1 X, a_4 Y a_1^{-1} Y \rangle
\]
whose star graph $\Gamma_{22}$ is given by Figure 4.8(iv).  Assigning in $\Gamma_{22}$ weight 1 to $e_1$, $e_2$ and $b_1$ and weight 0 to the remaining edges yields an aspherical weight function.

\medskip
Finally let $a_1=a_2=a_3^{-1}$ in which case $a_4$ is isolated.  Then $\mathcal{P}_X$ is equivalent to
\[
\mathcal{P}_{X,23} \colon \langle A \ast B,X,Y \mid Y^{-1} Xa_1 X^{-1} b_1,
Y^2 b_1^{-1} b_2 b_1 Y^{-1} b_3 X a_4 X^{-1} b_4 \rangle
\]
whose star graph $\Gamma_{23}$ is given by Figure 4.8(v).  Assigning in $\Gamma_{23}$ weight 1 to $e_1$ and $a_4$, weight 0 to $a_1$ and $b_1^{-1} b_2 b_1$ and weight $\frac{1}{2}$ to the remaining four edges yields an aspherical weight function unless $b_3b_4=1$, in which case assigning weight 1 to $b_3$ and $b_4$, weight 0 to $b_1^{-1} b_2 b_1$ and $e_2$ and weight $\frac{1}{2}$ to the remaining four edges yields an aspherical weight function unless $b_1 (b_1^{-1} b_2 b_1)^m = 1$ for some $m$.  But $|m|>1$ forces $b_1$ to be isolated so it remains to consider $b_1 b_2^{\pm 1}=1$.  If
$b_3b_4=b_1b_2=1$ then $\mathcal{P}_X$ is equivalent to
\[
\mathcal{P}_{X,24} \colon \langle A \ast B,X,Y \mid Y^{-1} X^{-1} b_3^{-1} Xa_1 X^{-1} b_1 X, Ya_1 Y^{-1} a_4 \rangle
\]
whose star graph $\Gamma_{24}$ is given by Figure 4.8(vi).  Assigning in $\Gamma_{24}$ weight 1 to $e_1$, $e_2$ and $b_1$ and weight 0 to the remaining edges yields an aspherical weight function.  Or if $b_3b_4=b_1b_2^{-1}=1$ then $\mathcal{P}_X$ is equivalent to
\[
\mathcal{P}_{X,25} \colon \langle A \ast B,X,Y \mid Y^{-1}Xa_1^{-1} X^{-1} b_3 X,
Ya_4 Y^{-1} b_1 Xa_1 X^{-1} b_1 \rangle
\]
whose star graph $\Gamma_{25}$ is given by Figure 4.8(vii) in which $b_{1,1}=b_{1,2}=b_1$.  Assigning a $\Gamma_{25}$ weight 1 to each $a_1$ edge and $b_{1,1}$, weight $\frac{1}{2}$ to $e_2$ and $b_3$ and weight 0 to $a_4$, $e_1$ and $b_{1,2}$ yields an aspherical weight function.

\medskip
 This completes the proof of Theorem 1.


\end{document}